\documentclass[10pt]{amsart}
\usepackage{graphicx}
\usepackage{amsthm}
\usepackage{amsfonts}
\usepackage{amsmath}
\usepackage{amssymb}
\usepackage{enumerate}
\usepackage{dsfont}
\usepackage{mathrsfs}

\DeclareMathOperator{\NVT}{NV_t}
\DeclareMathOperator{\NV}{NV_0}
\DeclareMathOperator{\A}{A}
\DeclareMathOperator{\Leb}{Leb}

\newcommand{\dist}{\textrm{dist}}
\newcommand{\dep}{\textrm{dep}}
\newcommand{\ess}{\textrm{ess}}
\newcommand{\supp}{\textrm{supp}}
\newcommand{\R}{\mathbb{R}}
\newcommand{\eps}{\varepsilon}
\newcommand{\Crit}{\mathcal{C}}
\newcommand{\tB}{\widetilde{B}}
\newcommand{\SC}{\textrm{SC}}
\newcommand{\PR}{\textrm{PR}}
\newcommand{\CV}{\textrm{CV}}

\newcommand{\sC}{\mathscr{C}}

\newcommand{\hS}{\widehat{S}}

\newcommand{\tS}{\widetilde{S}}

\newtheorem{lem}{Lemma}[section]

\newtheorem{defi}{Definition}[section]

\newtheorem{pro}{Proposition}[section]
\newtheorem{cor}{Corollary}[section]

\newtheorem*{coro}{Corollary}

\newtheorem*{mainthm}{Main Theorem}
\newtheorem*{reducedmain}{Reduced Main Theorem}

\numberwithin{equation}{section}

\begin{document}
\title{Summability implies Collet-Eckmann almost surely}
\author{Bing Gao and Weixiao Shen}
\date{\today}
\address{Bing Gao, Block S17, 10 Lower Kent Ridge Road, Singapore~119076, Singapore}
\email{g0901232@nus.edu.sg}
\address{Weixiao Shen, Block S17, 10 Lower Kent Ridge Road, Singapore~119076, Singapore}
\email{matsw@nus.edu.sg}
\maketitle

\begin{abstract}
We provide a strengthened version of the famous Jakobson's theorem. Consider an interval map $f$ satisfying a summability condition. For a generic one-parameter family $f_t$ of maps with $f_0=f$, we prove that $t=0$ is a Lebesgue density point of the set of parameters for which $f_t$
satisfies both the Collet-Eckmann condition and a strong polynomial recurrence condition.
\end{abstract}

\section{Introduction}
The famous result of Jakobson~\cite{J} states that maps with stochastic behavior are abundant, in the probabilistic sense,
in one-dimensional dynamics, which opened the way to much progress in non-uniformly expanding dynamics.
Several alternative proofs and generalizations of this result were obtained in subsequent works, see~\cite{BC1, BC2, R, TTY, T, V, Lu, Ly1, Yo, AM05, WT} among others.
In this paper, we shall provide another generalization of this result.

To state our result, we start with some definitions. Let $f:[0,1]\to [0,1]$ be a $C^1$ map and let $\Crit(f)$ denote the set of critical points of $f$.
We say that $f$ satisfies {\em the summability condition } (abbreviated ($\SC$)), if for each $c\in \Crit(f)$,
we have
$$\sum\limits_{n=0}^{\infty}\frac{1}{|Df^n(f(c))|}<\infty.$$
We say $f$ satisfies {\em the Collet-Eckmann condition } (abbreviated (CE)), if for each $c\in \Crit(f)$, we have
$$\liminf\limits_{n\to\infty}\frac{1}{n}\log|Df^n(f(c))|>0.$$
Furthermore, we say $f$ satisfies the {\em polynomial recurrence condition of exponent $\beta$} (abbreviated ($\PR_\beta$)), if there exists $C>0$ such that for any $c, c'\in\Crit(f)$ and any $n\ge 1$, we have
$$\dist (f^n(c), c')\ge Cn^{-\beta}.$$
If for each $\beta>1$, $f$ satisfies $\PR_{\beta}$, then we say that $f$ satisfies the {\em strong polynomial recurrence condition} (abbreviated (SPR)).

Let $\mathcal{A}$ be the collection of $C^1$ interval maps $f: [0,1]\to [0,1]$ with the following properties:
\begin{itemize}
\item $f$ has no attracting or neutral periodic orbits;
\item each critical point of $f$ lies in the interior $(0,1)$;
\item $f$ is $C^3$ outside $\Crit(f)$;
\item for each critical point $c$, there exist $\ell>1$ and a $C^3$ diffeomorphism $\varphi: \R\to \R$ such that $\varphi(c)=0$
    and such that $|f(x)-f(c)|= |\varphi(x)-\varphi(c)|^\ell$ holds near $c$.
\end{itemize}

Consider a one-parameter $C^1$ family $f_t: [0,1]\to [0,1]$, $t\in [-\delta,\delta]$, with $f_0\in\mathcal{A}$. We say that this family is {\em regular}
if the following hold:
\begin{enumerate}
\item The map $(t,x)\mapsto f_t(x)$ is $C^2$ on $\{(t,x)\in [-\delta,\delta]\times [0,1]: f_t'(x)\not=0\}$.
\item There exist $C^2$ functions $c_i: [-\delta,\delta]\to (0,1)$, $i=1,2,\ldots, d$, such that $0<c_1(t)<c_2(t)<\ldots <c_d(t)<1$ and $\mathcal{C}(f_t)=\{c_i(t):1\le i\le d\},$
\item For each $1\le i\le d$, there exist $\ell_i>1$, $\eps>0$ and a $C^2$ family $\varphi_t$ of diffeomorphisms of $\mathbb{R}$ such that $\varphi_t(c_i(t))=0$, and $|f_t(x)-f_t(c_i(t))|= |\varphi_t(x)-\varphi_t(c_i(t))|^{\ell_i}$ holds when $|x-c_i(t)|<\eps$ and $|t|\le \delta$. 
\end{enumerate}

It is easy to see that if $f_t:[0,1]\to [0,1]$, $t\in [-1,1]$, is a $C^3$  family such that $f_0\in\mathcal{A}$ has only non-degenerate critical points, then for $\delta>0$ small enough, $\{f_t\}_{|t|<\delta}$ is a regular family. Besides, if $f_t, t\in [-1,1]$, is a real analytic family such that all the maps $f_t$ have the same number of critical points, and the corresponding critical points have the same order, then $f_t$ is regular.

For a measurable subset $X$ of $\mathbb{R}^n$, let $\Leb_n(x)$ denote the Lebesgue  measure of $X$ in $\mathbb{R}^n$.
For simplify, let $|X|$ denote the $\Leb_1(X)$ for the measurable set $X\subset \mathbb{R}$.

\begin{mainthm}
Consider a regular one-parameter family $f_t:[0,1]\to [0,1]$, $t\in [-1,1]$ and denote $F(x, t)=f_t(x)$. Assume
\begin{itemize}
\item $f_0$ satisfies the summability condition (SC);
\item the following non-degeneracy condition holds for $t=0$.
$$(\NVT) \qquad\sum\limits_{j=0}^{\infty}\frac{\partial_tF(f_t^j(c),t)}{Df_t^j(f_t(c))}\neq 0
\text{ for any critical point $c\in\mathcal{C}(f_t)$}.$$
\end{itemize}
Define $$\mathscr{Z}:=\{t\in [-1,1]:f_t \text{ satisfies (CE), (SPR) and ($NV_t$)}\}.$$
Then we have
$$\lim\limits_{\varepsilon\to 0}\frac{\left|[-\eps,\eps]\cap \mathscr{Z}\right|}{2\varepsilon}=1.$$
In particular, $|\mathscr{Z}|>0$.
\end{mainthm}

Like most of the approaches to the Jakobson's theorem, our proof is purely real analytic. Comparing to the previous works, our assumption on $f_0$ is much weaker and the result on strong polynomial recurrence condition is new. Previously the weakest assumption was given in~\cite{T}, where $f_0$ satisfies (CE) and the critical points are at most sub-exponentially recurrent. Our analysis on the phase space geometry is based on the recent work~\cite{S} by the second author, and these estimates  are transformed to the parameter space by modifying the argument in~\cite{T}.

For the family of real quadratic polynomials, our theorem is implicitly contained in~\cite{AM05}, where complex method developed in~\cite{Ly1} was applied to relate the phase and parameter spaces.
The complex method is powerful for uni-critical maps, but does not work for multimodal maps.

The non-degeneracy condition ($\NVT$) was introduced in \cite{T}. In~\cite{AM03}, a geometric interpretation of this condition was given: for a real analytic family $f_t$ of unimodal maps for which $f_0$ satisfies ($\SC$), $(\NVT)$ holds at $t=0$ if and only $f_t$ is transversal to the topological conjugacy class of $f_0$. In~\cite{Le1} and~\cite{A}, it was proved that for the family of quadratic maps $Q_t(z)=z^2+t$, if $Q_{t_0}$ satisfies ($\SC$) then the condition $(\NVT)$ automatically holds at $t=t_0$. By~\cite{Ly2}, for almost every $t\in\R$, $Q_t$ is either uniformly hyperbolic or satisfies ($\SC$). Thus our theorem gives a new proof of Theorem A and a part of Theorem B in~\cite{AM05}.

Recently this transversality result has been generalized to higher degree polynomials in~\cite{Le2}. With this result, we can extend our Main Theorem to the high dimensional version. More precisely, for any integer $n\in\mathbb{N}$, let $\textbf{a}=(a_1,a_2,\cdots,a_n)\in\mathbb{C}^n$ and $P_{\textbf{a}}=\sum\limits_{i=1}^n a_i x^i+\Big(1-\sum\limits_{i=1}^n a_i\Big) x^{n+1}$. Hence, $P_{\textbf{a}}(0)=0$ and $P_{\textbf{a}}(1)=1$.

Let $\mathcal{P}$ be the collection of polynomial maps $P$ with the following properties:
\begin{itemize}
\item $P([0,1])\subset [0,1]$;
\item $P|_{[0,1]}\in\mathcal{A}$ and $P|_{[0,1]}$ satisfies (SC).
\end{itemize}

\begin{coro}
Fix $n\in\mathbb{N}$, we define parameter sets $$\Lambda=\{\textbf{a}\in\mathbb{R}^n: P_{\textbf{a}}\in\mathcal{P}\},$$ and
$$\Lambda_0=\left\{\textbf{a}\in\Lambda:P_{\textbf{a}}|_{[0,1]}\text{ satisfies (CE) and (SPR) conditions }\right\}.$$
Then we have $\Leb_n(\Lambda\backslash \Lambda_0)=0$.
\end{coro}

\begin{proof}
Consider parameter set $\Lambda_1=\{\textbf{a}\in\mathbb{R}^n: P_{\textbf{a}}\text{ has degenerate critical points}\}$.
For any $\textbf{a}\in\Lambda_1$, the discriminant $\Delta(\textbf{a})$ of $P_{\textbf{a}}'$ is equal to zero. Since $\Delta(\textbf{a})$ is a polynomial in $\textbf{a}$, the set $\Lambda_1$ has codimension one in $\R^n$, hence $\Leb_n(\Lambda_1)=0$.

Define 
$$\Pi=\left\{\textbf{a}\in\mathbb{C}^n: \text{ all critical points of $P_{\textbf{a}}$ are non-degenerate}\right\}.$$
Fix $\textbf{a}_*\in \Pi$. For $\textbf{a}$ in a small neighborhood of $\textbf{a}_*$, the critical points of $P_{\textbf{a}}$   $c_1(\textbf{a}),c_2(\textbf{a}),\cdots,c_n(\textbf{a})$ depend on $\textbf{a}$ analytically.
Letting $v_j(\textbf{a})=P_{\textbf{a}}(c_j(\textbf{a}))$ for $j=1,2,\cdots,n$, by Proposition $1$ in~\cite{Le3}, $\{v_1(\textbf{a}),v_2(\textbf{a}), \cdots,v_n(\textbf{a})\}$ is a local analytic coordinate.

Now let $\textbf{a}_*\in\Lambda\backslash\Lambda_1$.
Suppose $c_1,c_2,\cdots,c_r$ be the all critical points of $P_{\textbf{a}_*}$ in $(0,1)$. By Theorem 1 in~\cite{Le2}, the rank of matrix
$$\textbf{L}=(L(c_j,v_k))_{1\leq j\leq r,1\leq k \leq n}$$ is equal to $r$,
where
$$L(c_j,v_k):=\lim\limits_{m\to \infty}\frac{\frac{\partial P_{\textbf{a}}^m(c_j)}{\partial v_k}\big|_{\textbf{a}=\textbf{a}_*}}{(P_{\textbf{a}_*}^{m-1})'(P_{\textbf{a}_*}(c_j))}.$$
Notice that $\{a_1,a_2,\cdots,a_k\}$ is a globe analytic coordinate, then we define
$$L(c_j,a_k):=\lim\limits_{m\to \infty}\frac{\frac{\partial P_{\textbf{a}}^m(c_j)}{\partial a_k}\big|_{\textbf{a}=\textbf{a}_*}}{(P_{\textbf{a}_*}^{m-1})'(P_{\textbf{a}_*}(c_j))}.$$
Hence, the rank of matrix $$\widehat{\textbf{L}}=(L(c_j,a_k))_{1\leq j\leq r, 1\leq k\leq n}$$ is equal to $r$ and all entries of $\widehat{\textbf{L}}$ are real numbers.

For any direction $\textbf{u}\in S^{n-1}$, let $F^{(\textbf{u})}(x,t):=P_{\textbf{a}_*+t\textbf{u}}(x)$,
then we have
$$\sum\limits_{m=0}^{\infty}\frac{\partial_t F^{(\textbf{u})}(P_{\textbf{a}_*}^m(c_j),0)}{DP_{\textbf{a}_*}^m(P_{\textbf{a}_*}(c_j))}=
(L(c_j,a_1),L(c_j,a_2),\cdots,L(c_j,a_k))\cdot \textbf{u}.$$ Thus, ($\NV$) condition holds for
$F^{(\textbf{u})}(x,t)$ if and only if all entries of $\widehat{\textbf{L}}\cdot\textbf{u}$ are nonzero. Since the rank of matrix $\widehat{\textbf{L}}$ is equal to $r$, all rows of matrix $\widehat{\textbf{L}}$ are nonzero. If the $k$-th entry of $\widehat{\textbf{L}}\cdot\textbf{u}$ is equal to $0$, then $\textbf{u}$ is contained in the intersection of hyperplane in $\mathbb{R}^n$ and
$S^{n-1}$. Thus, for almost all $\textbf{u}$ in $S^{n-1}$ (endowed with the Lebesgue measure on $S^{n-1}$), all entries of $\widehat{\textbf{L}}\cdot\textbf{u}$ are nonzero.

Hence, for almost every direction $\textbf{u}$ in $S^{n-1}$, $(\NV)$ condition holds for one-parameter family $F^{(\textbf{u})}(x,t)$. Together with our Main Theorem, it follows that $\textbf{a}_*$ is a density point of set $\Lambda_0$ along line $\textbf{a}_*+ t \textbf{u}$. By proposition 5.2 in~\cite{AM03}, $\Leb_n((\Lambda\backslash\Lambda_1)\backslash\Lambda_0)=0$. Then the statement follows.

\end{proof}


The paper is organized as follows. In \S~\ref{sec:reduction}, we state a Reduced Main Theorem from which we deduce the Main Theorem. The rest of the paper is devoted to the proof of the Reduced Main Theorem. As described by Adrien Douady, the proof consists of two steps: in \S~\ref{sec:plough} we ``plough in the phase space'' and in \S~\ref{sec:harvest} we ``harvest in the parameter space''.

\noindent
{\bf Acknowledgment.} BG is supported a Research Scholarship from NUS and WS is supported by Research Grants R-146-000-128-133
and C-146-000-027-001 from NUS. We would like to thank G. Levin for helpful conversations on his non-degeneracy result.

\section{Reduction}\label{sec:reduction}

\subsection{Normalization}\label{subsec:regularfamliy}
A regular family $g_t:[0,1]\to [0,1]$, $t\in [-1, 1]$, 
is called {\em normalized} if the following hold:
\begin{enumerate}
\item[(i)] the maps $g_t$, $t\in [-1,1]$, all have the same critical points (denoted by $c_1, c_2, \ldots, c_d$);
\item[(ii)] there exists $\eps_*>0$ and for each $i\in \{1,2,\ldots,d\}$ there exists $\ell(c_i)>1$ such that $|g_t(x)-g_t(c_i)|=|x-c_i|^{\ell(c_i)}$ holds for all
$t\in[-1,1]$ and $x\in B(c_i,\eps_*)$;
\item[(iii)] $|\partial_t G(x,t)|\le 1$ for all $x\in [0,1]$ and $|t|\le 1$.
\end{enumerate}

To prove the Main Theorem, we only need to consider a normalized regular family. Indeed,
given any regular family $f_t: [0,1]\to [0,1]$, $t\in [-1,1]$, one can find a $C^2$ family $h_t$, $t\in [-1,1]$, of diffeomorphisms from $[0,1]$ onto itself, such that $\widetilde{f}_t=h_t\circ f_t\circ h_t^{-1}$ all have the same critical points and are normalized as in (ii).
Furthermore, take a small constant $\kappa$ and define $g_t= \widetilde{f}_{t\kappa}$. Then the family $G(x,t)=g_t(x)$, $t\in [-1,1]$, satisfies all the properties (i), (ii), (iii).
Note also that if $f_0$ satisfies (SC) then $g_0$ satisfies (SC); and if $F$ satisfies the non-degeneracy condition $(\NVT)$ at $t=0$, then so does $G$.

In the remaining of this paper, we assume that $F$ is a normalized regular family. Let $\Crit$ denote the common set of critical points of $f_t$, and let $$\ell_{\max}=\max\{\ell(c): c\in \Crit\}\text{ , }\ell_{min}=\min\{\ell(c):c\in\Crit\}.$$
Moreover, let $f=f_0$ and $\CV=f(\Crit)$.

\subsection{The (CE) and (PR) conditions}
For each $c\in\mathcal{C}$ and $\delta>0$, let
$$\widetilde{B}(c;\delta)=B(c,\delta^{1/\ell(c)}), \,\, D_c(\delta)=\frac{\delta}{|\tB(c;\delta)|}=\frac{1}{2}\delta^{1-\ell(c)^{-1}},$$
and let
$$\widetilde{B}(\delta)=\bigcup\limits_{c\in\mathcal{C}}\widetilde{B}(c;\delta).$$

The space $\{f_t\}_{t\in [-1,1]}$ is admissible in the sense of ~\cite{S}. Thus by~\cite[Theorem 1]{S}, we have the following:
\begin{pro}\label{prop:lemma-lambda}
For each $\varepsilon>0$ small enough,
there exist $\Lambda(\varepsilon)>1$ and $\alpha(\varepsilon)>0$ such that
$$\lim_{\varepsilon\to 0}\Lambda(\varepsilon)=\infty\text{, }\lim\limits_{\varepsilon\to 0}\alpha(\varepsilon)= 0$$
and the following hold for $|t|<\varepsilon$:
\begin{enumerate}
\item[(i)]Let $x\in[0,1]$ be such that $\dist(x,\CV)\leq 4\varepsilon$,
with $f_t^j(x)\notin\widetilde{B}(\varepsilon)$ for $j=0,1,\cdots,s-1$
and $f_t^s(x)\in\widetilde{B}(c;2\varepsilon)$ for some $c\in\mathcal{C}$. Then
\begin{equation}
|Df_t^s(x)|\geq \frac{\Lambda(\varepsilon)}{D_c(\varepsilon)}\exp(\eps^{\alpha(\varepsilon)}s)
\end{equation}
\item[(ii)] Let $x\in[0,1]$ be such that $f_t^j(x)\notin\widetilde{B}(\varepsilon)$ for $j=0,1,\cdots, s-1$, then
\begin{equation}
|Df_t^s(x)|\geq A \varepsilon^{1-\ell_{\max}^{-1}}\exp(\eps^{\alpha(\varepsilon)}s)
\end{equation}
where $A>0$ is a constant independent of $\varepsilon$.
\end{enumerate}
\end{pro}

\noindent
{\bf Remark.} This result is our starting point to prove abundance of Collet-Eckmann parameters near a summable one.  The proof is based on decomposition of an $f_t$-orbit into pieces that can be shadowed by $f_0$-orbits and a delicate choice of the binding periods played an central role.

Let $\mathbb{N}=\{0,1,\ldots\}$ denote the set of natural numbers.
Define
\begin{equation}\label{eqn:qeps}
q_\eps(x)=
\inf\left\{k\in\mathbb{N}: x\not\in \bigcup_{c\in\Crit}\tB(c;e^{-k\ell(c)}\eps)\right\}.
\end{equation}
Note that for $x\in\tB(c;\eps)$ with $\eps>0$ small and $c\in\Crit$, we have
\begin{multline}\label{eqn:Dfq}
|Df_t(x)|=\ell(c)\dist(x,\Crit)^{\ell(c)-1}\ge \ell(c) \left(e^{-q_\eps(x)\ell(c)} \eps\right)^{1-\ell(c)^{-1}}\\
> e^{-q_\eps(x)(\ell(c)-1)}D_c(\eps) > e^{-q_\eps(x) \ell_{\max}} D_c(\eps).
\end{multline}
Thus the following is an immediate consequence of Proposition~\ref{prop:lemma-lambda}.
\begin{lem}\label{lem:Dfnq}
Provided that $\eps>0$ is small enough, the following holds: For any $y\in \tB(\eps)$, $t\in [-\eps,\eps]$, and $n\ge 1$,
putting
$$m=\#\{1\le k\le n: f_t^k(y)\in\tB(\eps)\},$$
we have
\begin{equation}
|Df_t^n(f_t(y))|\ge A\eps^{1-\ell_{\max}^{-1}} \Lambda(\eps)^m \exp \left(-\ell_{\max} \sum_{k=1}^n  q_\eps(f_t^k(y))\right) e^{\eps^{\alpha(\eps)} n}.
\end{equation}
Furthermore, if $f_t^n(y)\in \tB(\eps)$, then
\begin{equation}
|Df_t^n(f_t(y))|\ge \Lambda(\eps)^m \exp \left(-\ell_{\max} \sum_{k=1}^n  q_\eps(f_t^j(y))\right).
\end{equation}
\end{lem}
\begin{proof}
Let $0=n_0<n_1<\cdots< n_m$ be all the integers in $\{0,1,\ldots, n\}$ such that $f_t^{n_j}(y)\in \tB(\eps)$.
Note that $\dist(f_t(y),\CV)\le 2\eps$. Applying Proposition~\ref{prop:lemma-lambda} (i) to obtain lower bounds for
$|Df_t^{n_{j+1}-n_j-1}(f_t^{n_j+1}(y))|$, $0\le j< m$, applying (ii) to obtain lower bounds for $|Df_t^{n-n_m-1} (f_t^{n_m+1}(y))|$ in the case $n_m<n$,
and applying (\ref{eqn:Dfq}) give us the desired inequalities.
\end{proof}
For $t\in[-1,1]$, $\varepsilon>0$ and $c\in\Crit$,
let $S_1^{(c)}(t;\varepsilon)<S_2^{(c)}(t;\varepsilon)<\cdots<S_n^{(c)}(t;\varepsilon)<\cdots$
be the all positive integers such that $f_t^{S_j^{(c)}(t;\varepsilon)+1}(c)\in\widetilde{B}(\varepsilon)$, and let
\begin{equation}\label{eqn:dfnd}
d^{(c)}_j(t;\varepsilon)=q_\eps (f_t^{S_j^{(c)}(t;\eps)+1}(c)).
\end{equation}

\noindent
{\bf Convention.} If $c$ returns to $\tB(\eps)$ at most $n-1$ times, then let $S_n^{(c)}(t;\eps)=\infty$ and
$d^{(c)}_n(t;\eps)=0.$

Given $C>0$, for each $n=1,2,\ldots,$ we define
\begin{equation}
\label{eqn:setXn}
X_{n,\varepsilon}(C)  =
\bigg\{t\in[-\eps,\eps]:\sum\limits_{j=1}^{k}d^{(c)}_j(t;\varepsilon)\leq C k
\text{ for any $k< n$ and } c\in\Crit\bigg\},
\end{equation}
and
\begin{equation}
X_{\varepsilon}(C)=\bigcap\limits_{n=1}^{\infty}X_{n,\varepsilon}(C).
\end{equation}

Given $C>0$ and $\tau>0$, for each $m=0,1,\ldots,$ we define
\begin{equation}
\label{eqn:setYn}
Y^m_\eps (C,\tau) =\left\{ t\in X_\varepsilon(C):
\dist (f_t^{k+1}(\Crit), c) \ge \frac{\varepsilon^{1/\ell(c)}}{ (k+1)^{\tau}} \textrm{ for } 0\le k< m \mbox{ and } c\in \Crit\right\}
\end{equation}
and

\begin{equation}
Y_\eps(C,\tau)=\bigcap\limits_{m=0}^{\infty}Y^m_{\eps}(C,\tau).
\end{equation}

\begin{lem}\label{lem:condCEPR}
\begin{enumerate}
\item[(i)] Given $C>0$, the following holds provided that $\eps>0$ is small enough: for any $t\in X_\eps(C)$, $f_t$ satisfies the condition (CE).
\item[(ii)] For any $C>0, \tau>1$ and $\alpha>0$, if $t\in Y_\eps(C,\tau)$ and $\eps>0$ is small enough, then $f_t$ satisfies the condition $\PR_{\tau}$.
\end{enumerate}
\end{lem}

\begin{proof}
For any $t\in X_\eps(C)$, $n\geq 1$ and any $c\in\Crit$, let $$m=\#\{1\le k\le n: f_t^k(c)\in\tB(\eps)\}.$$
By lemma~\ref{lem:Dfnq}, we have
\begin{multline*}
|Df_t^n(f_t(c))|\ge A\eps^{1-\ell_{max}^{-1}} e^{\eps^{\alpha(\eps)} n} \Lambda(\eps)^m
\exp \left(-\ell_{max}\sum_{j=1}^m d^{(c)}_j(t;\eps)\right)\\
\ge A\eps^{1-\ell_{max}^{-1}} e^{\eps^{\alpha(\eps)} n} \left(\Lambda(\eps) e^{-\ell_{\max} C}\right)^m
\ge A\eps^{1-\ell_{max}^{-1}} e^{\eps^{\alpha(\eps)} n},
\end{multline*}
provided that $\eps>0$ is small enough so that $\Lambda(\eps)\ge e^{\ell_{\max} C}$.
Hence, $f_t$ satisfies the condition (CE).
The second statement is trivial.
\end{proof}

\begin{reducedmain}
Let $F=(f_t)$ be a normalized regular family of interval maps. Assume that $f_0$ satisfies ($\SC$) and that the condition $(\NVT)$ holds at $t=0$. Then
\begin{enumerate}
\item[(i)] Given $C>0$ there exists $K=K(C)>0$ such that $K(C)\to \infty$ as $C\to\infty$ and such that
$$\left|X_{n,\eps}(C)\setminus X_{n+1,\eps}(C)\right|\le K^{-n} \eps, \,\, n=1,2,\ldots,$$
provided that $\eps>0$ is small enough.
\item[(ii)] Given $C>0$, the following holds provided that $\eps>0$ is small enough:
for any $t\in X_\eps(C)$, $(\NVT)$ holds.
\item[(iii)] Given $C>0$, $\tau>\tau_0>1$, and $\sigma>0$,  we have
$$\left|Y^m_{\eps}(C,\tau)\setminus Y^{m+1}_\eps (C,\tau) \right|\le \sigma \eps (m+1)^{-\tau_0}, \,\, m=0,1,,\ldots,$$
provided that $\eps>0$ is small enough.
\end{enumerate}
\end{reducedmain}

\begin{proof}[Proof of the Main Theorem]
For $\tau>1$, let $\mathscr{Z}_\tau$ denote the set of parameters $t\in [-1,1]$ for which $f_t$ satisfies the conditions (CE), $(\NVT)$ and $(\PR_{\tau})$. We shall prove that
\begin{equation}\label{eqn:mathscrZtau}
\lim_{\eps\to 0} \frac{\left|\mathscr{Z}_\tau\cap [-\eps,\eps]\right|}{2\eps}=1.
\end{equation}

Fix $\tau_0\in (1,\tau)$ and $\eta>0$. Choose a large constant $C>0$ and a small constant $\sigma>0$, such that
$$\frac{2K-3}{K-1}-\sigma \sum_{m=0}^\infty (m+1)^{-\tau_0}> 2-\eta,$$
where $K=K(C)$ is as in the Reduced Main Theorem.
Provided that $\eps>0$ is small enough, we have
$$|X_\eps(C)|=|X_{1,\eps}(C)|-\sum_{n=1}^\infty |X_{n,\eps}(C)\setminus X_{n+1,\eps}(C)|\ge \frac{2K-3}{K-1} \cdot\eps,$$
and
\begin{equation*}
|Y_\eps(C,\tau)|\ge |X_{\eps}(C)|-\sum_{m=0}^\infty \left|Y^m_{\eps}(C,\tau)\setminus Y^{m+1}_{\eps}(C,\tau) \right|
\ge \left(2-\eta\right)\eps.
\end{equation*}
By Lemma~\ref{lem:condCEPR}, and the second statement of the Reduced Main Theorem, we have $Y_\eps(C,\tau)\subset \mathscr{Z}_\tau$. Thus
$$|\mathscr{Z}_\tau\cap [-\eps,\eps]|\ge (2-\eta) \eps.$$
The equality (\ref{eqn:mathscrZtau}) follows.

To complete the proof, we shall show that $\mathscr{Z}_2\setminus \mathscr{Z}$ has zero measure. Since $\mathscr{Z}=\bigcap_{k=1}^\infty \mathscr{Z}_{1+k^{-1}}$, we only need to show that for each $\tau>1$, $\mathscr{Z}_2\setminus \mathscr{Z}_\tau$ has measure zero. Indeed, for each $t_0\in \mathscr{Z}_2$ and $\tau>1$, we can apply the above argument to $f_{t_0}$ instead of $f_0$, and obtain
that $t_0$ is not a Lebesgue density point of $\mathscr{Z}_2\setminus \mathscr{Z}_\tau$. By Lebesgue density Theorem, the statement follows.
\end{proof}

\noindent
{\bf Notations.} We collect the notations which will be used in the rest of the paper.
For each $c\in\Crit$, let
$$W^{(c)}=\sum_{n=0}^\infty |Df^n(f(c))|^{-1}<\infty$$
and $$a^{(c)}=\sum_{n=0}^\infty\frac{\partial_t F(f^n(c),0)}{Df^n(f(c))}\not=0.$$

For any $x\in[0,1]$, $t\in[-1,1]$ and $n\in\mathbb{N}$, we define
$$\A(x,t,n)=\sum\limits_{j=0}^{n-1}\frac{|Df_t^j(x)|}{\dist (f_t^j(x),\mathcal{C})}.$$
So if $f_t^j(x)\in\mathcal{C}$ for some $j<n$, then $\A(x,t,n)=\infty$.

As before let $S_j^{(c)}(t;\eps)$ denote
the $j$-th return time of $f_t(c)$ into $\tB(\eps)$, let $d_j^{(c)}(t;\eps)$ be as defined in (\ref{eqn:dfnd}). Define
$$P_{j}^{(c)}(t;\eps)=\frac{|Df_t^{S_j^{(c)}(t;\eps)}(f_t(c))|}{\dist(f_t^{S_j^{(c)}(t;\eps) +1}(c),\Crit)},$$
$$p_j^{(c)}(t;\eps)=\log \frac{|P_j^{(c)}(t;\eps)|}{\A(f_t(c), t, S_j^{(c)}(t;\eps))},$$
and
$$\widetilde{p}_j^{(c)}(t;\eps)=\min\left\{p_j^{(c)}(t;\eps),d_j^{(c)}(t;\eps) \right\}.$$
\section{Ploughing in the phase space}\label{sec:plough}
In this section, we obtain some estimates in the phase space.
The main results are Propositions~\ref{lemma-sum-bounded} and~\ref{prop:totaldepth} below.
Lemmas~\ref{lem:psorbretpre},~\ref{lemma-local-bounded-distortion} and~\ref{lem:theta0} used in the argument are taken from~\cite{S}.
Note that the non-degeneracy condition ($\NVT$) plays no role in this section.
\subsection{A uniform summability}

\begin{pro}\label{lemma-sum-bounded}
Given $\delta>0$,  the following holds provided that $\eps>0$ is small enough. For any $t\in [-\eps,\eps]$, $c\in\Crit$, $x\in [0,1]$ with $\dist(x, f(c))\le 4\eps$,
if $n$ is a non-negative integer such that $f_t^j(x)\not\in \tB(\eps)$ holds for all $0\le j<n$, then
$$\sum_{j=0}^{n} |Df_t^j(x)|^{-1}\le W^{(c)}+\delta.$$
\end{pro}

Before we prove this proposition, let us state a corollary.
\begin{cor}\label{cor:Xnepsbound} Given $\theta>0$ and $C>0$ the following holds provided that $\eps>0$ is small enough: for each $t\in X_{n,\eps}(C)$ and $c\in\Crit$, if $S_n$ is the $n$-th return time of $f_t(c)$ into $\tB(\eps)$, then
$$\sum_{i=0}^{S_n}|Df_t^i(f_t(c))|^{-1}\le W^{(c)} +\theta. $$
\end{cor}
\begin{proof} Denote $W=\max\limits_{c\in\Crit}W^{(c)}$ and fix constants $\delta\in (0,\theta)$ and $\Lambda> (W+\theta)/(\theta-\delta)$.
Let $S_0=-1$, and for each $j\ge 1$, let $S_j$ be the $j$-th return time of $f_t(c)$ into $\tB(\eps)$. Write $y_j=f_t^{S_j+1}(c)$, $x_j=f_t(y_j)$.
Provided that $\eps>0$ is small enough, by Proposition~\ref{lemma-sum-bounded}, we have
$$H_k:=\sum_{i=0}^{S_{k+1}-S_k-1} |Df_t^i(x_k)|^{-1}\le W+\delta,$$
for each $k=1,2,\ldots$.
Moreover, by Lemma~\ref{lem:Dfnq}, we have $|Df_t^{S_k+1}(f_t(c))|\ge \Lambda^k$.
Thus
$$\sum_{i=0}^{S_n}|Df_t^i(f_t(c))|^{-1} \leq W^{(c)}+\delta  + \sum_{k=1}^{n-1}|Df_t^{S_k+1}(f_t(c))|^{-1}H_k< W^{(c)}+\theta.$$
\end{proof}

Fix $\eps_0>0$ small such that Propositions~\ref{prop:lemma-lambda} holds for all $\eps\in (0,4\eps_0]$ with $\Lambda(\eps)\ge 4$.
For each $\eps,\eps'\in (0,4\eps_0]$ and $c\in\Crit$, let $\mathcal{D}^{(c)}(\eps,\eps')$
be the collection of all triples $(x, t, n)$ with the following properties:
$|x-f(c)|\leq 4 \eps '$,  $|t|\le \eps$, and $n$ is a non-negative integer such that $f_t^j(x)\not\in \tB(\eps)$ for all $0\le j<n$, and let
$$\widehat{L}^{(c)}(\eps,\eps')= \sup\left\{\sum_{i=0}^n |Df_t^i(x)|^{-1}: (x,t, n)\in\mathcal{D}^{(c)}(\eps,\eps')\right\}.$$

Moreover, let
$$L^{(c)}(\eps)=\widehat{L}^{(c)}(\eps,\eps),$$
$$L^{(c)}_*(\eps)=\sup\{{L}^{(c)}(\eps'): \eps'\in [\eps, 4\eps_0]\},\,\,L_*(\eps)=\max_{c\in\Crit} L^{(c)}_*(\eps),$$
and  $$\widehat{L}(\eps,\eps')=\max_{c\in\Crit}\widehat{L}^{(c)}(\eps,\eps').$$
By Proposition~\ref{prop:lemma-lambda} (ii), $1\le L_*(\eps)<\infty$ for each $\eps>0$.

\begin{lem}\label{lem:anyentereps}
For any $0<\eps\leq\eps'\leq 2\eps_0$, we have
\begin{equation}\label{eqn:L*gehL}
\widehat{L}(\eps,\eps')\leq 4 L_*(\eps)\left(\frac{\eps'}{\eps}\right)^{1-\ell_{max}^{-1}}
\end{equation}
\end{lem}

\begin{proof}
It suffices to prove that for any integer $k\geq 0$ such that $2^k\eps \leq 4\eps_0$, we have
\begin{equation}\label{eqn:inductionK}
\widehat{L}(\eps,2^k\eps)\leq 2 L_*(\eps) \cdot 2^{k(1-\ell_{max}^{-1})}.
\end{equation}
Indeed, this implies that for any $\eps'\in[2^{k-1}\eps,2^k\eps]$, we have
\begin{equation}\label{eqn:induciton}
\widehat{L}(\eps,\eps')\leq \widehat{L}(\eps,2^{k}\eps)\leq 2 L_*(\eps)\cdot 2^{k(1-\ell_{max}^{-1})}
< 4 L_*(\eps)\left(\frac{\eps'}{\eps}\right)^{1-\ell_{max}^{-1}}.
\end{equation}

Let us prove (\ref{eqn:inductionK}) by induction on $k$. By definition, the case $k=0$ is clear. Assume (\ref{eqn:inductionK})
holds for all $k$ not greater than some $j$. Let us consider the case $k=j+1$ with $2^{j+1}\eps\leq 4\eps_0$.
For $c\in\Crit$ and $(x,t,n)\in \mathcal{D}^{(c)}(\eps,2^{j+1}\eps)$, we need to prove that
\begin{equation}\label{eqn:hLepsb}
\sum\limits_{i=0}^n |Df_t^i(x)|^{-1}\leq 2 L_*(\eps)\cdot 2^{(j+1)(1-\ell_{max}^{-1})}.
\end{equation}

If $f_t^i(x)\notin\widetilde{B}(2^{j+1}\eps)$ holds for all
$0\leq i<n$, then $(x,t,n)\in\mathcal{D}^{(c)}(2^{j+1}\eps,2^{j+1}\eps)$, so (\ref{eqn:hLepsb}) holds by definition of $L_*$. Otherwise, let
$m\in\{1,2,\ldots, n-1\}$ be minimal such that $f_t^m(x)\in \widetilde{B}(2^{j+1}\eps)$. Then
$$\sum\limits_{i=0}^m |Df_t^i(x)|^{-1}\leq L_*(2^{j+1}\eps)\leq L_*(\eps).$$
Let $c_*\in\Crit$ be the critical point closest to $f_t^m(x)$ and $\eps_*=|f_t^{m+1}(x)-f_t(c_*)|$.
By Proposition~\ref{prop:lemma-lambda} (i), we have
$$|Df_t^{m+1}(x)|\geq \Lambda(2^{j+1}\eps)\left(\frac{\eps_*}{2^{j+1}\eps}\right)^{1-\ell(c_*)^{-1}}
\geq 4\left(\frac{\eps_*}{2^{j+1}\eps}\right)^{1-\ell_{max}^{-1}}.$$
Notice that $|f_t^{m+1}(x)-f(c_*)|\leq 2\eps_*$. If $\eps_*/2\leq \eps$, we have that
$$\sum\limits_{i=0}^{n-m-1}|Df_t^i(f_t^{m+1}(x))|^{-1}\leq L_*(\eps)\le
4L_*(\eps)\left(\frac{\eps_*}{2\eps}\right)^{1-\ell_{max}^{-1}}
,$$
where for the last inequality we have used $\eps_*\ge \eps$.
Otherwise, $\eps<\eps_*/2\leq 2^j\eps$, by induction and (\ref{eqn:induciton}), we have
$$\sum\limits_{i=0}^{n-m-1}|Df_t^i(f_t^{m+1}(x))|^{-1}\leq \widehat{L}(\eps,\eps_*/2)
\leq 4L_*(\eps)\left(\frac{\eps_*}{2\eps}\right)^{1-\ell_{max}^{-1}}.$$

Thus,
\begin{eqnarray*}
\sum\limits_{i=0}^n |Df_t^i(x)|^{-1}
&=&\sum\limits_{i=0}^m |Df_t^i(x)|^{-1}+|Df_t^{m+1}(x)|^{-1}\sum\limits_{i=0}^{n-m-1}|Df_t^i(f_t^{m+1}(x))|^{-1}\\
&< &  2L_*(\eps) 2^{(j+1)(1-\ell_{max}^{-1})}  .
\end{eqnarray*}
%
%
\end{proof}

To complete the proof, we shall need the following result which is a reformulation of~\cite[Proposition 5.2]{S}.

\begin{lem}\label{lem:psorbretpre}
For $\eps>0$ sufficiently small and each $c\in \Crit$, there exists a constant $\Lambda_0(\eps)>0$ and
a positive integer $M=M_c(\eps)\ge 1$ such that
$\lim_{\eps\to 0}\Lambda_0(\eps)=\infty$ and such that
the following holds:
for any $t\in [-\eps,\eps]$ and $y\in [0,1]$ with $|y- f(c)|\le 4 \eps$, we have
\begin{align}
\label{eqn:yjposition}
y_j :=f_t^j(y)& \not\in \tB(2\eps)\mbox{ for all }0\le j<M;\\
\label{eqn:der}
e^{-1} |Df^j(f(c))|& \le |Df_t^j(y)|\le e|Df^j(f(c))|\mbox{ for all } 0\le j\le M.
\end{align}
If $f_t^M(y)\notin\widetilde{B}(\eps_0)$, then
\begin{equation}\label{eqn:notineps_0}
|Df_t^{M+1}(y)|\geq \Lambda_0(\eps)\left(\frac{\eps_0}{\eps}\right)^{1-\ell_{max}^{-1}};
\end{equation}
If $f_t^M(y)\in\widetilde{B}(\eps_0)$ and $f_t^{M}(y)\notin\widetilde{B}(\eps)$, then
\begin{equation}\label{eqn:ineps_0}
|Df_t^{M+1}(y)|\geq \Lambda_0(\eps)\left(\frac{\dist(f_t^{M+1}(y),\CV)}{\eps}\right)^{1-\ell_{max}^{-1}}.
\end{equation}
\end{lem}

\begin{lem}\label{lem:recbound}
Let $\delta>0$ be given. Then for $\eps>0$ small enough, and any $c\in\Crit$,
$$L^{(c)}(\eps)\le W^{(c)} + \delta L_*(\eps).$$
\end{lem}
\begin{proof} In the following, we assume $\eps>0$ small.  We need to prove
that for each $(x,t,n)\in \mathcal{D}^{(c)}(\eps,\eps)$,
\begin{equation}\label{eqn:LleL*}
\sum_{j=0}^{n} |Df_t^j(x)|^{-1}\le W^{(c)}+\delta L_*(\eps).
\end{equation}
Let $M=M_c(\eps)$ be as in Lemma~\ref{lem:psorbretpre}. We first prove
\begin{equation}\label{eqn:bindbound}
\sum_{j=0}^{\min (n,M)} |Df_t^j(x)|^{-1}\le W^{(c)}+\delta L_*(\eps) /2.
\end{equation}
Take $N$ large enough such that
$$\sum_{j=0}^N |Df^j(f(c))|^{-1}\ge W^{(c)}-\delta/(4e).$$
By continuity, we have
\begin{equation}\label{eqn:continuitybound}
\sum_{j=0}^{\min (n,N)} |Df_t^j(x)|^{-1}\le W^{(c)}+\delta/4.
\end{equation}
So (\ref{eqn:bindbound}) holds when $\min(n,M)\le N$. If $\min(n,M)>N$, then
by (\ref{eqn:der}),
\begin{equation}\label{eqn:bindingbound}
\sum_{j=N+1}^{\min(n,M)} |Df_t^j(x)|^{-1}\le e \sum_{j=N+1}^M |Df^j(f(c))|^{-1}\le \delta/4\le \delta L_*(\eps)/4,
\end{equation}
since $L_*(\eps)\ge 1$.
Together with (\ref{eqn:continuitybound}), this implies (\ref{eqn:bindbound}).

In particular, (\ref{eqn:LleL*}) holds if $n\le M$. Let us assume now that $n>M$, so that
$f_t^M(x)\notin\widetilde{B}(\eps)$. To complete the proof, we need to prove that
\begin{equation}\label{eqn:M+12n}
\sum_{j=M+1}^n |Df_t^j(x)|^{-1}\le \delta L_*(\eps)/2.
\end{equation}
We distinguish two cases.

\noindent
{\em Case 1.} Assume $f_t^M(x)\in\widetilde{B}(c_*;\eps_0)$ for some $c_*\in\Crit$. Let $\eps_*:=\dist(f_t^{M+1}(x),f(c_*))$. Then
$\eps_*\in [\eps, 2\eps_0]$ and $(f_t^{M+1}(x), t, n-M-1)\in \mathcal{D}^{(c_*)}(\eps,\eps_*)$.
By Lemma~\ref{lem:anyentereps},
we have
$$\sum\limits_{j=M+1}^n |Df_t^j(x)|^{-1}\leq \frac{\widehat{L}(\eps,\eps_*)}{|Df_t^{M+1}(x)|}\le  \frac{4L_*(\eps)}{|Df_t^{M+1}(x)|} \left(\frac{\eps_*}{\eps}\right)^{1-\ell_{max}^{-1}}.$$
Together with (\ref{eqn:ineps_0}), this implies that
$$\sum_{j=M+1}^n |Df_t^j(x)|^{-1}\le 4\Lambda_0(\eps)^{-1} L_*(\eps)<\delta L_*(\eps)/2.$$

\noindent
{\em Case 2.} Assume $f_t^{M}(x)\notin\widetilde{B}(\eps_0)$. Let $k$ be the maximal integer with $M<k\le n$ and such
that $f_t^j(x)\not\in\widetilde{B}(\eps_0)$ for all $M<j< k$. By Proposition~\ref{prop:lemma-lambda} (ii), there exists a constant
$C>0$ such that
\begin{equation}\label{eqn:sumM+12k}
\sum\limits_{j=0}^{k-M-1}|Df_t^j(f_t^{M+1}(x))|^{-1}\leq C
\end{equation}
\begin{equation}\label{eqn:derM+12k}
|Df_t^{k-M-1}(f_t^{M+1}(x))|\ge 1/C .
\end{equation}
Thus by (\ref{eqn:notineps_0}),
\begin{equation}\label{eqn:sumM+12k'}
\sum_{j=M+1}^k |Df_t^j(x)|^{-1}\le C |Df_t^{M+1}(x)|^{-1}\le \frac{C}{\Lambda_0(\eps)}\left(\frac{\eps}{\eps_0}\right)^{1-\ell_{\max}^{-1}}<\delta L_*(\eps)/4.
\end{equation}
In particular, (\ref{eqn:M+12n}) holds if $k=n$. Assume that $k<n$. Then there exists $c_*\in\Crit$ such that $f_t^k(x)\in \tB(c_*;\eps_0)$. Let $\eps_*:=\dist(f_t^{k+1}(x), f_t(c_*))\in [\eps,\eps_0]$. Then  $$(f_t^{k+1}(x), t, n-k-1)\in \mathcal{D}^{(c_*)}(\eps,\eps_*).$$ So by Lemma~\ref{lem:anyentereps}
$$\sum_{j=k+1}^n |Df_t^j(x)|^{-1}\le |Df_t^{k+1}(x)|^{-1} 4L_*(\eps) \left(\frac{\eps_*}{\eps}\right)^{1-\ell_{\max}^{-1}}.$$
On the other hand,
$$|Df_t(f_t^k(x))|=\ell_{c_*} \eps_*^{1-\ell_{c_*}^{-1}}\ge \eps_*^{1-\ell_{\max}^{-1}},$$
so by (\ref{eqn:derM+12k}) and (\ref{eqn:notineps_0}),
\begin{align*}
|Df_t^{k+1}(x)|  = |Df_t^{M+1}(x)| |Df_t^{k-M-1}(f_t^{M+1}(x))| |Df_t(f_t^k(x))|
\ge \frac{\Lambda_0(\eps)\eps_0^{1-\ell_{\max}^{-1}}}{C} \left(\frac{\eps_*}{\eps}\right)^{1-\ell_{\max}^{-1}}.
\end{align*}
Therefore,
$$
\sum\limits_{j=k+1}^n|Df_t^j(x)|^{-1}\le \frac{4CL_*(\eps)}{\Lambda_0(\eps)\eps_0^{1-\ell_{\max}^{-1}}} \le \delta L_*(\eps) /{4}.
$$
Together with (\ref{eqn:sumM+12k'}), this implies (\ref{eqn:M+12n}). This completes the proof of the lemma.
\end{proof}


\begin{proof}[Completion of proof of Proposition~\ref{lemma-sum-bounded}]
By Lemma~\ref{lem:recbound}, it suffices to show that $L_*(\eps)$ is uniformly bounded. Arguing by contradiction, assume that this is not the case. As $L^{(c)}_*(\eps)$ is monotone decreasing in $\eps$ for each $c$,
it follows that $L_*(\eps)\to\infty$ as $\eps\to 0$. By definition of $L^{(c)}(\eps)$, this implies that there exists $\eps_k\to 0$ and $c\in\Crit$
such that $2 L^{(c)}(\eps_k)\ge L_*(\eps_k)$. However, by Lemma~\ref{lem:recbound},
we have $$L_*(\eps_k)\le 2 L^{(c)}(\eps_k)\le 2W^{(c)}+ \frac{1}{2} L_*(\eps_k),$$
provided that $k$ is large enough. It follows that $L_*(\eps_k)\le 4W^{(c)}$, a contradiction.
\end{proof}

\subsection{Essential returns}

\begin{defi}
We say that $S_n^{(c)}(t;\eps)$ is an {\em essential return time} of $f_t(c)$ into $\tB(\eps)$ if
$$P_{n}^{(c)}(t;\eps)\ge 3^{n-k} P_{k}^{(c)}(t;\eps), \text{ for all } 1\le k<n. $$
Given $C_0>0$, we define
$$\mathcal{T}_\ess^{(c)} (t;\eps)=\{k\ge 1: S_k^{(c)}(t;\eps)\text{ is an essential return time  of } f_t(c) \text{ into } \tB(\eps)\},$$
and
$$\widehat{\mathcal{T}}_\ess^{(c)}(C_0,t;\eps)=\{k\in\mathcal{T}_\ess^{(c)}(t;\eps): \widetilde{p}^{(c)}_k(t;\eps)>C_0\}.$$
\end{defi}
Refer to the end of section 2 for the definition of the notations $P_n^{(c)}$, $\widetilde{p}_k^{(c)}$, etc.

The goal of this section is to prove the following:
\begin{pro} \label{prop:totaldepth}
Given $C>0,C_0>0, \tau>1$ and $\gamma\in (0,1)$, the following hold provided that $\eps>0$ is small enough:
\begin{enumerate}
\item[(i)] For $t\in X_{n,\eps}(C)\setminus X_{n+1,\eps}(C)$, $n=1,2,\ldots$, there exists $c\in\Crit$ such that
$$\sum_{k\in \widehat{\mathcal{T}}_\ess^{(c)}(C_0,t;\eps), k\le n} \widetilde{p}^{(c)}_k(t;\eps)
\ge (\gamma C -C_0)n.$$
\item[(ii)] For $t\in Y^m_{\eps}(C,\tau)\setminus Y^{m+1}_{\eps}(C,\tau)$, $m=0,1,2,\ldots$, there exists $c\in\Crit$ and
$n\in \widehat{\mathcal{T}}_\ess^{(c)}(C_0,t;\eps)$ such that $m=S_n^{(c)}(t;\eps)$ and
$$p_n^{(c)}(t;\eps)\ge \gamma \tau \log (m+1).
$$
\end{enumerate}
\end{pro}

We shall need the following lemma which is~\cite[Proposition 5.6]{S}.
\begin{lem}\label{lemma-local-bounded-distortion}
For any $\varepsilon>0$ small enough, there exists a constant $\kappa(\varepsilon)>0$ such that for $|t|\le \eps$ and $x\in [0,1]$,
if $n$ is an integer such that $f_t^j(x)\notin\widetilde{B}(\eps)$ for $0\leq j <n $ and $f_t^n(x)\in\widetilde{B}(c;\eps)$ for some
$c\in\Crit$, then
\begin{equation}\label{eqa-local}
\A(x,t,n)\leq \kappa(\varepsilon)\cdot\frac{|Df_t^n(x)|}{|\widetilde{B}(c;\eps)|}
\end{equation}
and such that
\begin{equation}\label{eqa-kappa}
\kappa(\varepsilon)\to 0\text{ as }\varepsilon\to 0
\end{equation}
\end{lem}

We shall also need the following lemma which is~\cite[Lemma 2.1]{S}.
\begin{lem} \label{lem:theta0}
There exists a constant $\theta_0>0$ such that
for any $(x,t)\in [0,1]\times [-1,1]$ and any integer $n\ge 1$ with $\A(x,t, n)<\infty$, putting
$$J=\left[x-\frac{\theta_0} {\A(x,t, n)}, x+\frac{\theta_0}{ \A(x,t, n)}\right]\cap [0,1],$$ we have that $f_t^n|J$ is a diffeomorphism and
$$e^{-1} |Df_t^j(x)|\le |Df_t^j(y)|\le e \, |Df_t^j(x)|$$
holds for all $y\in J$ and $0\le j\le n.$
\end{lem}

In the following, fix $C>1$, $\gamma\in (0,1)$ and denote $\rho= 1-\sqrt{\gamma}$, $\rho_1=\rho/4$, $\rho_2=\rho_1/(2\ell_{max})$.
Let $\eps>0$ denote a small constant and we fix a parameter $t\in [-\eps,\eps]$.
For simplicity, we shall drop $t, \eps$ from the notations. So
$S_i^{(c)}=S_i^{(c)}(t;\eps)$, $d^{(c)}_i= d^{(c)}_i(t;\eps)$, etc.

\subsubsection{Free returns}
Define
$$\hS_i^{(c)}=\sup\{S>S_i^{(c)}: \A(f_t^{S_i^{(c)}+2}(c), t, S-S_i^{(c)}) \le \theta_0 e^{(d_i^{(c)}-1)\ell(c') }\eps^{-1}\},$$
and
$$\tS_i^{(c)}=\inf\{S> \hS_i^{(c)}: f_t^{S+1}(c)\in \tB(\eps)\},$$
where $c'$ denote the critical point of $f$ which is closest to $f_t^{S_i^{(c)}+1}(c)$.

\begin{lem}\label{lemma-estimate-depth}
Consider $t\in X_{n,\eps}(C)$, $c\in\Crit$ and $1\le i<n$.
Then \begin{equation}\label{equation-depth-bound}
\sum\limits_{k=S_i^{(c)}+2}^{\hS_i^{(c)}+1} q_\eps(f_t^{k}(c))<\rho_1\cdot d^{(c)}_i.
\end{equation}
Moreover, if there exists $j\le n$ such that $S_j^{(c)}=\tS_i^{(c)}$, then
\begin{equation}\label{eqn:estQ}
\log \frac{P_{j}^{(c)}}{P_{i}^{(c)}}> d^{(c)}_j -\rho_1 d^{(c)}_i+ (\log 3) \cdot(j-i).
\end{equation}
\end{lem}

\begin{proof}
Assume $\eps>0$ small and let $a=2\ell_{max}/(\ell_{min}-1)$, $\eps'=e^a \eps$.
Let $c_k$ denote the critical point of $f$ which is closest to $f_t^{S_k^{(c)}+1}(c)$.
For simplicity of notation, we shall write $S_k=S_k^{(c)}$, 
$\hS_k=\hS_k^{(c)}$ and $d_k=d_k^{(c)}$ for each $k$.
Let $y= f_t^{S_i+1}(c)$, $x=f_t(y)$, $v=f_t(c)$ and $v_i=f_t(c_i)$.
Note that $\A(x, t, \hS_i-S_i)\le \theta_0/|v_i-x|.$
So by Lemma \ref{lem:theta0}, for $0\le k< \hS_i-S_i$,
we have
\begin{equation}\label{eqn:derv2hatv}
e^{-1} |Df_t^{k+1}(x)|\le |Df_t^{k+1}(v_i)|\le e |Df_t^{k+1}(x)|,
\end{equation}
and
\begin{equation}\label{eqn:derlength}
|Df_t^{k+1}(x)|\ge e^{-1} \frac{\dist(f_t^{k+1}(v_i), f_t^{k+1}(x))}{|v_i-x|}.
\end{equation}

We shall first prove that
\begin{equation}\label{eqn:Mets}
M:=\#\{1\le k\le \hS_i-S_i: f_t^k(v_i)\in \tB(\eps')\}\le \rho_2 d_i (C+a+1)^{-1}<n.
\end{equation}
and
\begin{equation}\label{eqn:qesp'c'}
\sum_{k=0}^{\hS_i-S_i} q_{\eps'}(f_t^{k+1}(c_i))\le (C+a+1) M\le \rho_2 d_i.
\end{equation}

Indeed, $q_{\eps'}(z)\le q_\eps(z) +a+1$ holds for each $z\in [0,1]$, thus
\begin{equation}\label{eqn:Xeps}
t\in X_{n,\eps}(C)\subset X_{n,\eps'}(C+a+1).
\end{equation}
Therefore (\ref{eqn:qesp'c'}) will follow once we prove (\ref{eqn:Mets}).
Let $T_1<T_2<\cdots$ be all the positive integers such that $f_t^{T_k+1}(c_i)\in \tB(\eps')$ and $p_k$ be the critical point of $f$
which is closest to $f_t^{T_k+1}(c_i)$. Then
for each $1\le m< n$, by Lemma~\ref{lem:Dfnq}, we have
\begin{equation}\label{eqn:DfTm}
|Df_t^{T_m+1}(f_t(c_i))|\ge \left(\Lambda(\eps') e^{-\ell_{max}(C+a+1)}\right)^m,
\end{equation}
hence
\begin{multline*}
\A(v_i,t,T_m+1)  \ge \frac{|Df_t^{T_m}(f_t(c_i))|}{|f_t^{T_m+1}(c_i)-p_m|} =\frac{|Df_t^{T_m+1}(f_t(c_i))|}{|Df_t(f_t^{T_m+1}(c_i))||f_t^{T_m+1}(c_i)-p_m|}\\
\ge |Df_t^{T_m+1}(f_t(c_i))|(\eps')^{-1} \ge \left(\Lambda(\eps') e^{-\ell_{max}(C+a+1)}\right)^m (\eps')^{-1}.
\end{multline*}
%
On the other hand, by Lemma~\ref{lem:theta0}, we have $$\A(v_i,t, \hS_i-S_i)\asymp \A(x, t, \hS_i-S_i)\le \theta_0 e^{(d_i-1)\ell(c_i)} \eps^{-1}.$$
The inequality (\ref{eqn:Mets}) follows.

Let us now prove (\ref{equation-depth-bound}). Indeed, by (\ref{eqn:derv2hatv}), for each $0\le k<\hS_i-S_i$, we have
$|Df_t(f_t^k(v_i))|\le e^2 |Df_t(f_t^k(x))|$, so $q_{\eps'}(f_t^k(v_i))\ge q_{\eps}(f_t^k(x))$.
Thus
\begin{equation}\label{eqn:deptbdstrong}
\sum_{i<k<j} d_k\le \sum_{k=0}^{\hS_i-S_i} q_{\eps'}(f_t^{k+1}(c_i))\le \rho_2 d_i,
\end{equation}
which implies (\ref{equation-depth-bound}) since $\rho_2<\rho_1$.

To obtain (\ref{eqn:estQ}) it suffices to prove the following two inequalities:
\begin{equation}\label{eqn:DSjsimple}
|Df_t^{S_j-S_i-1}(x)|\ge \Lambda_1(\eps)^{j-i} \left(D_{c_j}(\eps)\right)^{-1},
\end{equation}
and
\begin{equation}\label{eqn:estimate-derivative}
|Df_t^{S_j-S_i-1}(x)|\ge \kappa \exp\left(\ell(c_i)d_i-\rho_2\ell_{max}d_i\right)\left(D_{c_j}(\eps)\right)^{-1},
\end{equation}
where $\Lambda_1(\eps)\to \infty$ as $\eps\to 0$ and $\kappa>0$ is a constant.

Indeed, combining these two inequalities, we obtain
$$U:=|Df_t^{S_j-S_i-1}(x)| D_{c_j}(\eps) \ge 3^{j-i} \exp \left(\ell(c_i) d_i-\rho_1 d_i+\ell_{\max}\right).$$
Since
$$\frac{P_j^{(c)}}{P_i^{(c)}}= U \frac{|Df_t(y)||y-c_i|}{|f_t^{S_j+1}(c)-c_j|D_{c_j}(\eps)}\ge U \exp\left(-\ell(c_i) d_i + d_j-1\right) ,$$
the inequality (\ref{eqn:estQ}) follows.

Let us prove (\ref{eqn:DSjsimple}). Applying Proposition~\ref{prop:lemma-lambda} (i),  we obtain
$$|Df_t^{S_{j}-S_{j-1}-1}(f_t^{S_{j-1}-S_i}(x))|\ge \Lambda(\eps)/D_{c_j}(\eps).$$
Thus (\ref{eqn:DSjsimple}) holds  with $\Lambda_1(\eps)=\Lambda(\eps)$ if $j=i+1$.
When $j>i+1$, $S_{j-1}-S_i$ is of the form $T_m+1$ for some $j-i-1\le m\le M<n$, so combining (\ref{eqn:DfTm}) with the last inequality, we obtain that (\ref{eqn:DSjsimple}) holds with a suitable choice of $\Lambda_1(\eps)$.

Finally let us prove (\ref{eqn:estimate-derivative}). We may certainly assume $(\ell_{max}-1)\rho_2d_i\ge 2$. Let
$$A_k=\frac{|Df_t^{S_{k}-S_i-1}(x)|}{\dist(f_t^{S_{k}-S_i-1}(x),\mathcal{C})}, \text{ and }
A_k'=\frac{|Df_t^{S_{k}-S_i-1}(x)|}{|\widetilde{B}(c_k;\eps)|}$$
for $i< k\le j$. Clearly, $A_k\ge A_k'$.
By Proposition~\ref{prop:lemma-lambda} (i), we have
$$\frac{A_j'}{A_k}=|Df_t^{S_j-S_k}(f_t^{S_k+1}(c))|\frac{\dist(f_t^{S_k+1}(c),\Crit)}{|\widetilde{B}(c_j;\eps)|}
\ge \Lambda(\eps)^{j-k}\exp \left(-\ell_{max}\sum_{k\le l<j} d_l \right),$$
which, by (\ref{eqn:deptbdstrong}), implies
\begin{equation}\label{eqn:Aj'vsAk}
\frac{A_j'}{A_k} \ge \Lambda(\eps)^{j-k} e^{-\rho_2 \ell_{max}d_i}.
\end{equation}
Let $\theta=\theta_0/(2e^{\ell_{max}})$. We distinguish two cases.

{\em Case 1.} Assume $\A(x, t, S_j-S_i-1)\ge \theta e^{d_i\cdot\ell(c_i)} \eps^{-1}.$
Then by Lemma~\ref{lemma-local-bounded-distortion}, we have
\begin{equation*}
\sum_{k=i+1}^{j-1} A_k+A_j'\ge \frac{1}{1+\kappa(\eps)}\A(x, t, S_j-S_i-1)\ge \theta e^{d_i\cdot\ell(c_i)} (2\eps)^{-1}.
\end{equation*}
Together with (\ref{eqn:Aj'vsAk}), this implies
$A_j'\ge \theta \exp(\ell(c_i)d_i-\rho_2\ell_{max}d_i)(4\eps)^{-1},$
provided that $\eps>0$ is small enough.
Thus (\ref{eqn:estimate-derivative}) holds in this case.

{\em Case 2.} Assume $\A(x, t, S_j-S_i-1)< \theta e^{\ell(c_i)\cdot d_i} \eps^{-1}.$
In particular we have $S_j-1\le\hS_i$ which implies $\hS_i=S_j-1$.
By maximality of $\hS_i$ we have
$$A_j=A(x, t, S_j-S_i) -A(x, t, \hS_i-S_i) \ge \theta e^{\ell(c_i)\cdot d_i} \eps^{-1}.$$
So (\ref{eqn:estimate-derivative}) holds if $d_j\le \rho_2\ell_{max} d_i$. Assume $d_j> \rho_2\ell_{max} d_i$. By (\ref{eqn:qesp'c'}),
$$q_{\eps'} (f_t^{S_j-S_i-1}(v_i))\le \rho_2 d_i\leq \rho_2 \ell_{max} d_i-2.$$ Thus there exists a constant $\kappa_1>0$ such that
$$\dist(f_t^{S_j-S_i-1}(v_i), f_t^{S_j-S_i-1}(x))\ge \kappa_1 e^{-\rho_2\ell_{max} d_i} |\widetilde{B}(c_j;\eps)|$$
Thus, by (\ref{eqn:derlength}),
\begin{eqnarray*}
|Df_t^{S_j-S_i-1}(x)|\ge e^{-1} \frac{\dist(f_t^{S_j-S_i-1}(v_i), f_t^{S_j-S_i-1}(x))}{|v_i-x|}
\ge\frac{\kappa_1  \exp(\ell(c_i)d_i-\rho_2\ell_{max} d_i)}{D_{c_j}(\eps)}.
\end{eqnarray*}
So the inequality (\ref{eqn:estimate-derivative}) holds.
\end{proof}

Given $c\in\Crit$, we define positive integers $i_1<i_2<\cdots$ in the following way:
$i_1=1$. Once $i_k$ and $S_{i_k}^{(c)}$ are both well-defined, let $i_{k+1}$ be such that
$\tS_{i_k}^{(c)}=S_{i_{k+1}}^{(c)}$. The procedure stops whenever $S_{i_k}^{(c)}$ or $\tS_{i_k}^{(c)}$ is not well-defined.
The positive integers $S_{i_k}, k=1,2,\ldots$ are called {\em free return times} of $f_t(c)$ into $\tB(\eps)$.

\begin{lem}\label{lem:deep2free}
An essential return time is a free return time.
\end{lem}
\begin{proof} By definition, for any consecutive free return times $S_i<S_j$, we have
$P_{k}<P_i$ for all $i<k<j$. So $S_k$ is not an essential return time. The lemma follows.
\end{proof}

\begin{lem} \label{lem:deepdepth}
Assume $t\in X_{n,\eps}(C)$. Let $c\in\Crit$. If $1\le i<j\le n$ are such that $S_i^{(c)}<S_j^{(c)}$ are consecutive essential return times of $f_t(c)$ into $\tB(\eps)$, then
\begin{equation}\label{eqn:dkij}
\sum_{i<k<j} d_k^{(c)} \le \rho d_i^{(c)},
\end{equation}
and
\begin{equation}\label{eqn:QjA}
 p_j^{(c)}\ge d_j^{(c)}-\rho d_i^{(c)}.
\end{equation}
Moreover, if $n_0$ is the largest integer in $\{1,2,\ldots, n\}$ such that $S_{n_0}^{(c)}$ is an essential return time of $f_t(c)$ into $\tB(\eps)$, then
\begin{equation}\label{eqn:dkn0}
\sum_{n_0<k\le n} d_k^{(c)}\le \rho d_{n_0}^{(c)}.
\end{equation}
\end{lem}
\begin{proof}
By Lemma~\ref{lem:deep2free}, $S_i^{(c)}$ and $S_j^{(c)}$ are both free return times. Let $i=k_0<k_1<\ldots<k_m=j$ be all the positive integers such that $S_{k_l}^{(c)}$ are free return times.
%
%
Then by Lemma~\ref{lemma-estimate-depth}, for each $0\le l<m$,
we have
\begin{equation}\label{eqn:QKl}
\log \frac{P^{(c)}_{k_{l+1}}}{P^{(c)}_{k_l}}\ge d_{k_{l+1}}^{(c)}-\rho_1 d_{k_{l}}^{(c)}+ (\log 3) (k_{l+1}-k_l),
\end{equation}
and
\begin{equation}\label{eqn:dkl}
\sum_{k_l<k<k_{l+1}} d_k^{(c)}\le \rho_1 d_{k_l}^{(c)}.
\end{equation}
Summing up both sides of (\ref{eqn:QKl}) for $0\le l<m-1$, we obtain
$$\log \frac{P^{(c)}_{k_{m-1}}}{P^{(c)}_{k_0}}
\ge d_{k_{m-1}}^{(c)}+(1-\rho_1) d_{k_{m-2}}^{(c)}+\cdots+ (1-\rho_1) d_{k_1}^{(c)}-\rho_1 d_{k_0}^{(c)}
+(\log 3) (k_{m-1}-k_0).$$
Since the left hand side is smaller than $(\log 3) (k_{m-1}-k_0)$, we obtain
$$d_{k_{m-1}}^{(c)}+(1-\rho_1) d_{k_{m-2}}^{(c)}+\cdots+ (1-\rho_1) d_{k_1}^{(c)}\le \rho_1 d_{k_0}^{(c)},$$
which, together with (\ref{eqn:dkl}), implies
$$\sum_{i<k<j} d_k^{(c)} \le\rho_1 d_i^{(c)}+ (1+\rho_1) \rho_1 (1-\rho_1)^{-1} d_i^{(c)}\le \rho d_i^{(c)}.$$
This proves (\ref{eqn:dkij}). Summing up both sides of (\ref{eqn:QKl}) for $0\le l<m$,
$$\log \frac{P^{(c)}_j}{P^{(c)}_i}\ge d^{(c)}_j+ (1-\rho_1) \sum_{l=1}^{m-1} d^{(c)}_{k_l} -\rho_1 d^{(c)}_{k_0} + (\log 3) (j-i)
\ge d^{(c)}_j- \rho_1 d_i^{(c)} + (\log 3) (j-i),$$
which is equivalent to
\begin{equation}\label{eqn:PivsPj}
P^{(c)}_i\le e^{-(d_j^{(c)}-\rho_1 d_i^{(c)})} 3^{i-j} P_j^{(c)}.
\end{equation}

Let us now prove that
\begin{equation}\label{eqn:PkvsPi}
P_k^{(c)}\leq 3^{k-i}P_i^{(c)}\mbox{ for any }1\leq k\leq j-1.
\end{equation}
Indeed, for $1\le k<i$, this inequality follows from that fact that $S_i^{(c)}$ is
an essential return times, while for $i<k<j$, it follows from the fact that $S_k^{(c)}$ is not an essential return time.

Combining (\ref{eqn:PivsPj}) and (\ref{eqn:PkvsPi}), we obtain
$$\sum_{k=1}^{j-1} P_k^{(c)}\le e^{-(d_j^{(c)}-\rho_1 d_i^{(c)})} 2^{-1} P_j^{(c)}.$$
By Lemma~\ref{lemma-local-bounded-distortion}, we have
$$\A(f_t(c), t, S_j)\le (1+\kappa(\eps))\sum_{k=1}^{j-1} P_k^{(c)} +\kappa(\eps)\frac{|Df_t^{S_j}(f_t(c))|}{|\widetilde{B}(c_j;\eps)|},$$
where $c_j$ be the critical point of $f$ which is closest to $f_t^{S_j^{(c)}+1}(c)$ and $\kappa(\eps)\to 0$ as $\eps \to 0$. So when $\eps>0$ is small, we obtain
$$\A(f_t(c), t, S_j)\le e^{-(d_j-\rho_1 d_i)} P_j^{(c)}.$$
The inequality (\ref{eqn:QjA}) follows.

The inequality (\ref{eqn:dkn0}) can be proved in a similar way.
\end{proof}

\begin{proof}[Proof of Proposition~\ref{prop:totaldepth}]
(i) By definition, there exists $c$ such that $\sum_{k=1}^n d_k^{(c)} \ge Cn.$
Let $i_1<i_2<\cdots<i_m$ be all the integers in $\{1,2,\ldots, n\}$ such that
$S_{i_j}$ is an essential return times of $f_t(c)$ into $\tB(\eps)$.  Then $i_1=1$.
For convenience of notations, we regard $i_0=0$ and $d^{(c)}_{i_0}=0$.

By (\ref{eqn:dkij}) and (\ref{eqn:dkn0}) in Lemma~\ref{lem:deepdepth}, we have
$$\sum_{j=1}^m d_{i_j}^{(c)} \ge (1-\rho) \sum_{j=1}^n d_j^{(c)} \ge (1-\rho)Cn.$$
Thus
$$\sum_{j=1}^m \big(d_{i_j}^{(c)} -\rho d_{i_{j-1}}^{(c)}\big) \ge (1-\rho)^2 Cn.$$

By (\ref{eqn:QjA}) in Lemma~\ref{lem:deepdepth},
for each $2\le j\le m$ we have
$$p_{i_j}^{(c)} \ge d_{i_{j}}^{(c)}-\rho d_{i_{j-1}}^{(c)}.$$
By Lemma~\ref{lemma-local-bounded-distortion}, this estimate is also true for $j=1$.
Thus $$\widetilde{p}^{(c)}_{i_j}\ge d_{i_{j}}^{(c)}-\rho d_{i_{j-1}}^{(c)}$$ holds for all $j=1,2,\ldots, m$,
which implies
$$\sum_{j=1}^m \widetilde{p}_{i_j}^{(c)} \ge \sum_{j=1}^m \big(d_{i_j}^{(c)} -\rho d_{i_{j-1}}^{(c)}\big) \ge (1-\rho)^2 Cn.$$
Consequently,
$$\sum\limits_{k\in\widehat{\mathcal{T}}_\ess^{(c)}(C_0,t;\eps),k\leq n}\widetilde{p}^{(c)}_k
\ge \sum\limits_{j=1}^m \widetilde{p}^{(c)}_{i_j} -C_0n
\ge (1-\rho)^2Cn -C_0n=(\gamma C-C_0)n.$$

(ii) By definition, there exists $c, c'\in\Crit$ such that $\dist(f_t^{m+1}(c), c')\le \eps^{1/\ell(c')} (m+1)^{-\tau}$. So there exists $n\ge 1$ such that $m=S_n^{(c)}$ and
$d_n^{(c)}\ge \tau \log (S_n^{(c)}+1)$. Since $d_k^{(c)}< d_n^{(c)}$ holds for each $1\le k<n$, by (\ref{eqn:dkij}) in Lemma~\ref{lem:deepdepth} it follows that
$n$ is an essential return time of $f_t(c)$ into $\tB(\eps)$ and hence
$p^{(c)}_n\ge (1-\rho)d^{(c)}_n>C_0$.
The statement is proved.
\end{proof}

\section{Harvest in the parameter space}\label{sec:harvest}
In this section, we transfer the estimates in phase space to the parameter space and prove the Reduced Main Theorem.
The phase and parameter spaces are related through the maps $\xi_n^{(c)}(t)=f_t^{n+1}(c)$.
In \S~\ref{subsec:parabox}, we define parameter boxes.  In \S~\ref{subsec:conclusion}, we prove the Reduced Main Theorem by showing that the bad parameters are contained in certain families of parameter boxes with large total depth. Proposition~\ref{lemma-sum-bounded} will be used to construct the parameter boxes and Proposition~\ref{prop:totaldepth} will be used to estimate the total depth.  The parameter boxes which we use are always mapped into $\tB(\eps)$ and they form special families of balls. In \S~\ref{subsec:specialfamily}, an abstract lemma about sets of points lying deeply in a special family of balls is proved.

\subsection{Parameter boxes}\label{subsec:parabox}
Recall that
$$\frac{(\xi_n^{(c)}(t))'}{Df_t^n(f_t(c))}=\sum\limits_{j=0}^n\frac{\partial_t F(f_t^j(c),t)}{Df_t^j(f_t(c))}=: M_n^{(c)}(t).$$

\begin{defi}
Given $m\ge 0$, $\lambda>1$ and $c\in\Crit$, we say that
a ball $B(t_0, r)$ in the parameter space is a {\em $\lambda$-bounded $c$-parameter box of order $m$} if
the following hold:
\begin{itemize}
\item $\xi_m^{(c)}: B(t_0, r)\to [0,1]$ is a diffeomorphism onto its image such that
$$\sup_{t, s\in B(t_0, r)}\frac{(\xi_m^{(c)})'(t)}{(\xi_m^{(c)})'(s)}\le \lambda.$$
\item For any $t\in B(t_0,r)$, we have $$\lambda^{-1} |a^{(c)}| \le \left|M_m^{(c)}(t)\right|\le \lambda|a^{(c)}|.$$
\item for each $k=0,1,\ldots, m$, we have
$$\sup_{t,s \in B(t_0,r)}\frac{|Df_t^{k}(f_t(c))|}{|Df_{s}^k (f_{s}(c))|}\le \lambda.$$
\end{itemize}
\end{defi}

The goal of this section is to provide an estimate of the size of a parameter box centered at a given parameter $t_0$.

\begin{pro}\label{prop:parabox}
Given $\lambda>1$, there exist $\theta>0$ and $N\ge 1$ such that the following holds. Let $|t_0|\le \theta$,
$c\in\Crit$ and $m> N$ be such that
$$\sum_{i=0}^m |Df_{t_0}^i(f_{t_0}(c))|^{-1}\le W^{(c)} +\theta,$$
then putting
$$r=\theta /\A(f_{t_0}(c), t_0, m),$$
$B(t_0,r)$ is a $\lambda$-bounded $c$-parameter box of order $m$.
\end{pro}

Write
$D_n^{(c)}(t)=Df_t^n(f_t(c)).$

\begin{lem}\label{lem:Mest}
Given $\lambda>1$ there exist $\eta=\eta(\lambda)>0$ and an integer $N=N(\lambda)\ge 1$ such that the following holds.
Let $t\in [-\eta,\eta]$, $c\in\Crit$ and let $m> N$ be a positive integer.
Assume
\begin{equation}\label{eqn:sumnot2big}
\sum_{i=0}^m |D_i^{(c)}(t)|^{-1}\le W^{(c)}+\eta.
\end{equation}
Then $$\lambda^{-1} |a^{(c)}|< |M_m^{(c)}(t)|< \lambda |a^{(c)}|.$$
\end{lem}
\begin{proof}
Take $\delta>0$ small. Let $N$ be large such that
$$\sum_{i=0}^N |Df^i(f(c))|^{-1}> W^{(c)}-\delta, \text{ and } \left|M_N^{(c)}(0)-a^{(c)}\right|<\delta.$$
By continuity, there exists $\eta_0>0$ such that for any $t\in [-\eta_0, \eta_0]$, we have
$$\sum_{i=0}^N |D^{(c)}_i(t)|^{-1}> W^{(c)}-\delta, \text{ and } \left|M_N^{(c)}(t)-a^{(c)}\right|<\delta.$$
Now let $\eta=\min (\delta, \eta_0)$. If (\ref{eqn:sumnot2big}) holds, then we have
$$\sum_{i=N+1}^m |D^{(c)}_i(t)|^{-1}<\delta +\eta\le 2\delta.$$
Since $|\partial_t F|\le 1$, it follows that
$$\left|M_m^{(c)}(t)-a^{(c)}\right|\le \left|M_N^{(c)}(t)-a^{(c)}\right|+\sum_{i=N+1}^m |D^{(c)}_i(t)|^{-1} <3\delta.$$
The desired inequality follows since $\min\limits_{c\in\Crit}|a^{(c)}|>0$.
\end{proof}
\begin{proof}[Proof of Proposition~\ref{prop:parabox}]
Fix $\lambda>1$. Let $\lambda_0=\lambda^{1/4}$ and let $\eta=\eta(\lambda_0), N=N(\lambda_0)$ be given by Lemma~\ref{lem:Mest}. Let $\theta\in (0,\eta/2)$ and $\lambda_1\in (1,\lambda_0)$ be such that
$$\lambda_1 (W^{(c)}+\theta) \le W^{(c)}+\eta,$$
holds for each $c\in\Crit$.

Now let $t_0, c, m$ be as in the assumption of this proposition.
Then by continuity, there exists a maximal $r_0\in (0,\theta]$ such that for each $t\in \overline{B(t_0,r_0)}$ and any $1\le j\le m$, we have
\begin{equation}\label{eqn:compD}
\frac{1}{\lambda_1}\le \frac{D_j^{(c)}(t)}{D_j^{(c)}(t_0)} \le \lambda_1.
\end{equation}
So
$$\sum_{i=0}^m |D^{(c)}_i(t)|^{-1}\le \lambda_1\sum_{i=0}^m |D^{(c)}_i(t_0)|^{-1} \le \lambda_1 (W^{(c)}+\theta) \le W^{(c)}+\eta,$$
which implies by Lemma~\ref{lem:Mest} that
$$\lambda_0^{-1} |a^{(c)}| \le |M_m^{(c)}(t)|\le \lambda_0 |a^{(c)}|.$$
It follows that $B(t_0,r_0)$ is a $\lambda$-bounded $c$-parameter box of order $m$. So
it suffices to prove that $\theta_0:=r_0\cdot\A(f_{t_0}(c), t_0, m)$ is bounded away from zero. To this end, we only need to show that
there exists a constant $C>0$ such that for each $1\le j\le m$,
\begin{equation}\label{eqn:relateDtheta}
\left|\log \frac{D_j^{(c)}(t)}{D_j^{(c)}(t_0)}\right| \le C\theta_0.
\end{equation}
Indeed, if $r_0=\theta$ then $\theta_0\ge r_0= \theta$, and if $r_0<\theta$, then by maximality of $r_0$, there exists $t_1\in \overline{B(t_0,r_0)}$ and $j\in \{1,2,\ldots, m\}$ such that
\begin{equation}\label{eqn:extrem}
\text{either } D_{j}^{(c)}(t_1)=\lambda_1 D_j^{(c)}(t_0) \text{ or }D_{j}^{(c)}(t_0)=\lambda_1 D_j^{(c)}(t_1).
\end{equation}
Thus (\ref{eqn:relateDtheta}) implies that
$\theta_0\ge \min\left(\log\lambda_1/C, \theta\right).$

Let us prove (\ref{eqn:relateDtheta}).  First note that there exists $C_1>0$ such that for each $1\le j\le m$ and any $t\in\overline{B(t_0, r_0)}$,
we have $|M_j^{(c)}(t)|\le C_1$, so $$|(\xi_j^{(c)}(t))'|=|M_j^{(c)}(t) D^{(c)}_j(t)|\le C_1 |D^{(c)}_j(t)|.$$
Since $F$ is a normalized regular family, there exists a constant $C_2>0$ such that
$$\left|\frac{\partial^2 F}{\partial t\partial x}(x,t)\right|\le C_2 |Df_t(x)|,$$ 
for all $(x,t)$. (Indeed, for $x$ close to $\Crit$ the left hand side of this inequality is zero.) By (\ref{eqn:compD}), for each $0\le i<m$,
$$\frac{|Df_t(\xi^{(c)}_{i}(t))|}{|Df_{t_0}(\xi^{(c)}_{i}(t_0))|}
=\frac{|D^{(c)}_{i+1}(t)|}{|D^{(c)}_{i+1}(t_0)|}\frac{|D^{(c)}_{i}(t_0)|}{|D^{(c)}_{i}(t)|}\in [\lambda^{-1}, \lambda].$$
By non-flatness of the critical points, it follows that there exist constants $C_3$ and $C_4$ such that
$$\frac{|D^2 f_t(\xi_i^{(c)}(t))|}{|Df_t(\xi^{(c)}_{i}(t))|}
\le C_3\frac{|D^2 f_{t_0}(\xi_i^{(c)}(t_0))|}{|Df_{t_0}(\xi^{(c)}_{i}(t_0))|}\le \frac{C_4}{\dist(\xi_{i}^{(c)}(t_0), \Crit)}.$$
Since $\dist(\xi_{i}^{(c)}(t_0), \Crit)\le 1$,
and $|D_i^{(c)}(t)|\geq \lambda_1^{-1} |D_i^{(c)}(t_0)|$ is bounded away from zero,
there exists a constant $C_5>0$ such that
\begin{multline*}
\left|\frac{d Df_t(\xi_i^{(c)}(t))}{dt}\right|=\left|\frac{\partial^2 F}{\partial t\partial x} (\xi_i^{(c)}(t),t)+ D^2f_t(\xi_i^{(c)}(t)) (\xi_i^{(c)}(t))'\right|
\\\le C_2 |Df_t(\xi_i^{(c)}(t))|+C_4C_1 |Df_t(\xi_i^{(c)}(t)|\frac{|D^{(c)}_i(t)|} {\dist(\xi_{i}^{(c)}(t_0), \Crit)}\le C_5 |Df_t(\xi_i^{(c)}(t)|\frac{|D^{(c)}_i(t)|}{\dist(\xi_{i}^{(c)}(t_0), \Crit)}. 
\end{multline*}
Thus
\begin{multline*}
\left|\frac{ d \log |D_j^{(c)}(t)|}{dt}\right|= \left|\sum_{i=0}^{j-1} \frac{d Df_t(\xi_i^{(c)}(t))/dt}{Df_t(\xi_i^{(c)}(t))}\right|
\\
\le C_5\sum_{i=0}^{j-1}\frac{|D_i^{(c)}(t)|}{\dist(\xi_{i}^{(c)}(t_0), \Crit)}
\le C\A(f_{t_0}(c), t_0, j),
\end{multline*}
where $C=C_5\lambda_1$.
Since
$$\log \frac{D^{(c)}_j (t)}{D^{(c)}_j(t_0)}=\int_{t_0}^t  d \log |D_j^{(c)}(t)|,$$
the inequality (\ref{eqn:relateDtheta}) follows.
\end{proof}

\subsection{Special family of balls}\label{subsec:specialfamily}
Given $B=B(a,r)$ and $x\in\R$, we define
$$ \dep (x|B) =
   \left\{
  \begin{array}{rl}
    \inf\{k\in \mathbb{N}: |x-a|\ge e^{-k} r\}, & \hbox{if $|x-a|< e^{-2}r$;} \\
    0, & \hbox{otherwise.}
  \end{array}
\right.  $$
Moreover, for each $k\in\mathbb{Z}$, let
\begin{equation}
B^{(k)}=B(a, e^{-k}r).
\end{equation}

A finite family $\mathcal{M}=\{B_i=B(a_i,r_i)\}_{i\in\mathcal{I}}$ is called {\em special} if the following holds:
For any $i, j\in\mathcal{I}$, if $a_i\in B_j^{(1)}$ then there exists $k=k(i,j)\ge 1$ such that
$B_i\subset B_j^{(k-1)}\setminus B_j^{(k+1)}.$
In particular, the centers $a_i, i\in\mathcal{I}$ are pairwise distinct.

Given a special family as above, define
$$\mathcal{I}_0=\{i\in\mathcal{I}: \text{ for any } j\in\mathcal{I}, j\not=i, \text{ we have }a_i\not\in B_j^{(1)}\},$$
and for each $k\ge 1$, let
$$\mathcal{I}_k=\left\{i\in \mathcal{I}\setminus \bigcup\limits_{m=0}^{k-1}\mathcal{I}_{m}: \text{ for any } j\in \mathcal{I}\setminus \bigcup\limits_{m=0}^{k-1}\mathcal{I}_{m}, j\not=i, \text{ we have }a_i\not\in B_j^{(1)}\right\}.$$
The minimal integer $n\ge 0$ for which $\mathcal{I}_n=\emptyset$ is called the {\em height} of $\mathcal{M}$.
The support of $\mathcal{M}$ is defined as the union of all the elements of $\mathcal{M}$.

We shall use the next lemma to estimate measure of sets of bad parameters.
\begin{lem}\label{lem:measureest}
For each $0<\kappa<1$ there exists $K=K(\kappa)>1$ such that if $\mathcal{M}=\{B_i\}_{i\in\mathcal{I}}$ is a special family of height at most $n$ and
$$X_{\mathcal{M}}(N)=\left\{x\in \supp(\mathcal{M}): \sum_{i\in\mathcal{I}} \dep (x| B_i)\ge N\right\}, N=0,1,\ldots,$$
then
$$\left|X_{\mathcal{M}}(N)\right|\le K^n e^{-(1-\kappa)N} \left|\supp (\mathcal{M})\right|.$$
\end{lem}
\begin{proof}
Fix $0<\kappa<1$ and let $K=K(\kappa)= e^{5}/(1-e^{-\kappa})$. We shall prove the lemma by induction on the height $n$.
We take the trivial case $n=0$ for the starting step.
Now let $n_0$ be a positive integer and assume that the lemma holds for $n<n_0$.
Let us consider the case $n=n_0$.
Let $\mathcal{I}_0$ be defined as above, and let $\mathcal{I}'=\mathcal{I}\setminus \mathcal{I}_0$.
Let $$q_0(x)=\sum_{i\in \mathcal{I}_0} \dep (x|B_i),\,\,  q'(x)=\sum_{i'\in\mathcal{I}'} \dep (x|B_{i'}).$$
For each $q_0\ge 0$, and $q'\ge 0$, let $V(q_0)=\{x\in \supp(\mathcal{M}): q_0(x)=q_0\}$ and
$U(q_0, q')=\{x\in V(q_0): q'(x)\ge q'\}.$
Let us prove that
\begin{equation}\label{eqn:Xq0q'}
\left|U(q_0, q')\right|\le e^{-q_0+5} K^{n-1} e^{-(1-\kappa) q'} \left|\supp(\mathcal{M})\right|.
\end{equation}

To this end, we first note that the balls $B_i^{(2)}$, $i\in\mathcal{I}_0$, are pairwise disjoint. Thus
\begin{equation}\label{eqn:disVq0}
V(q_0)\subset \bigcup_{i\in\mathcal{I}_0} B_i^{(q_0-1)},
\end{equation}
and
for each $k=0,1,\ldots,$ we have
\begin{equation}\label{eqn:Vq0}
\sum_{i\in\mathcal{I}_0} |B_i^{(k)}|\le e^{-k+2} |\supp (\mathcal{M})|.
\end{equation}
In particular, $|V(q_0)|\le e^{-q_0+3}|\supp (\mathcal{M})|$, so
 the inequality (\ref{eqn:Xq0q'}) holds when $q'=0$. Assume now that $q'>0$ and 
let
$$\mathcal{M}'_{q_0}=\{B_{i'}: i'\in\mathcal{I}', B_{i'}^{(2)}\cap V(q_0)\not=\emptyset\}.$$
Then $\mathcal{M}'_{q_0}$ is a special family of height $<n$ and $U(q_0,q')\subset X_{\mathcal{M}_{q_0}'}(q')$. By the induction hypothesis, we have
\begin{equation}\label{eqn:Uq0q'}
\left|U(q_0,q')\right|\le \left|X_{\mathcal{M}'_{q_0}}(q')\right|\le K^{n-1} e^{-(1-\kappa) q'} \left|\supp (\mathcal{M}'_{q_0})\right|.
\end{equation}
Next, let us show
\begin{equation}\label{eqn:suppM'}
\supp (\mathcal{M}'_{q_0}) \subset \bigcup_{i\in\mathcal{I}_0} B_i^{(q_0-3)}.
\end{equation}
In fact, since $\mathcal{M}'_{q_0}\subset \mathcal{M}$,  (\ref{eqn:suppM'}) holds when $q_0\le 3$. Assume $q_0>3$.  By (\ref{eqn:disVq0}), for each $B_{i'}\in\mathcal{M}'_{q_0}$ there exists $i\in\mathcal{I}_0$ such that $B_{i'}^{(2)}\cap B_i^{(q_0-1)}\not=\emptyset$. By definition of special family, we have $B_{i'}\subset B_i^{(q_0-3)}$. Thus (\ref{eqn:suppM'}) holds.
By (\ref{eqn:Vq0}), it follows that
$$\big| \supp (\mathcal{M}'_{q_0}) \big|\le e^{-q_0+5}|\supp (\mathcal{M})|,$$
which, together with (\ref{eqn:Uq0q'}), implies (\ref{eqn:Xq0q'}).


Now let us complete the induction step. Fix $N\ge 0$ and
for each $q_0\ge 0$, let $q_0'=\max(N-q_0, 0)$. So $q_0+q_0'\ge N$. Since
$$X_{\mathcal{M}}(N)\subset \bigcup_{q_0=0}^\infty U(q_0,q_0'),$$
by (\ref{eqn:Xq0q'}), we obtain
\begin{eqnarray*}
|X_{\mathcal{M}}(N)|\le \sum_{q_0= 0}^\infty  |U(q_0,q_0')|&\le& K^{n-1}e^5 |\supp(\mathcal{M})| \sum_{q_0=0}^{\infty} e^{-\kappa q_0} e^{-(1-\kappa)N}\\
&=& K^n e^{-(1-\kappa)N} |\supp(\mathcal{M})|.
\end{eqnarray*}
This completes the proof.
\end{proof}

\subsection{Proof of the Reduced Main Theorem}\label{subsec:conclusion}
For $c\in\Crit$ and $m\ge 0$, let $\sC_m^{(c)}$ denote the set of parameters $t\in [0,1]$ for which the following hold:
\begin{itemize}
\item $f_t^{m+1}(c)\in\Crit$;
\item $f_t^j(c)\cap \Crit=\emptyset$ for all $1\le j\le m$.
\end{itemize}
A $c$-parameter-box $B(t,r)$ of order $m$ is called {\em pre-critical} if $t\in\sC_m^{(c)}$.

For $m\ge 0$, $c\in \Crit$, $t_*\in \sC_m^{(c)}$ and $\lambda>1$, let
$r_{\lambda}(t_*,\eps)$ be the maximal number $r$ which satisfy the following properties:
\begin{enumerate}
\item[(i)] $r\in (0,\eps]$ and $\xi^{(c)}_m(B(t_*, r))\subset \tB(\eps)$;
\item[(ii)] $B(t_*,r)$ is a $\lambda$-bounded pre-critical $c$-parameter boxes of order $m$.
\end{enumerate}
Given a positive integer $n$,
let $$\mathcal{M}_{n,\lambda}^{(c)}(\eps)=\left\{B(t_*, r_{\lambda}(t_*,\eps))\mid
\begin{matrix}
t_* \in \sC_m^{(c)}\text{ for some } m\ge 0
\text{ and } \\\text{ there exists } t\in B(t_*, r_\lambda(t_*,\eps))\text{ such that }\\
\#\{0\le j\le m: f_t^{j+1}(c)\in \tB(\eps)\}\le n.
\end{matrix}
\right\}.$$

\begin{lem} \label{lem:parapuzzlecombinatorics}
There exists $\lambda>1$ such that for each $c\in\Crit$, each $n\ge 1$ and each $\eps>0$ small,
$\mathcal{M}_{n,\lambda}^{(c)}(\eps)$ is a special family of height at most $n$.
\end{lem}
\begin{proof}
Assume that $\lambda>1$ is very close to $1$. To prove that $\mathcal{M}_{n,\lambda}^{(c)}(\eps)$ is special, let $B_i=B(t_i,r_i)$, $i=1,2,$ be distinct parameter boxes in $\mathcal{M}_{n,\lambda}^{(c)}(\eps)$, of order $m_i$, such that $t_1\in B_2^{(1)}$. We need to prove that $|B_1|/|t_1-t_2|$ is small. Let $c_1, c_2\in\Crit$ be such that
$\xi_{m_i}^{(c)}(B_i)\subset \tB(c_i;\eps)$.
Since $f_{t_1}^{j+1}(c)\not\in \Crit$ for each $0\le j\le m_2$, we have $m_1>m_2$.
By the bounded distortion property of $\xi_{m_2}^{(c)}|B_2$, it suffices to show that
$$\sup_{t\in B_1} \frac{|\xi_{m_2}^{(c)}(t)-\xi_{m_2}^{(c)}(t_1)|}{|\xi_{m_2}^{(c)}(t_1)-c_2|}$$ is sufficiently small.
This is clear: for each $t\in B_1$, and for each $0\le k\le m_2+1\le m_1$,
$$\lambda^{-1}|Df_{t_1}^{k}(f_{t_1}(c))|\le |Df_t^{k}(f_t(c))|\le \lambda|Df_{t_1}^{k}(f_{t_1}(c))|,$$
hence
$\lambda^{-1}|Df_{t_1}(\xi_{m_2}^{(c)}(t_1))|\le |Df_t(\xi_{m_2}^{(c)}(t))| \le \lambda |Df_{t_1}(\xi_{m_2}^{(c)}(t_1))|,$
so the statement follows from the local behavior of $f$ near $c_2$.

Let us prove that the height of $\mathcal{M}_{n,\lambda}^{(c)}(\eps)$ does not exceed $n$. Otherwise, there would exist $B_j\in \mathcal{M}_{n,\lambda}^{(c)}(\eps)$, $0\le j\le n$, such that $B_{n}\subsetneq B_{n-1}\subsetneq \cdots \subsetneq B_0$. Let $m_j$ be the order of $B_j$. Then as above, we would have $m_0<m_1<\cdots< m_n$. Then for $t\in B_n$,
$\{0\le j\le m_n: f_t^{j+1}(c)\in \tB(\eps)\}\supset \{m_0, m_1,\ldots, m_n\}$ would contain at least  $n+1$ elements, a contradiction.
%
\end{proof}

Now we fix a constant $\lambda>1$ so that the conclusion of Lemma~\ref{lem:parapuzzlecombinatorics} holds.
As before, we use $S_j^{(c)}(t;\eps)$ to denote the $j$-th return time of $f_t(c)$ into $\tB(\eps)$.
Then by Proposition~\ref{prop:parabox}, we have

\begin{lem} \label{lem:shiftparabox}
There exists a constant $C_0>0$ such that for any $C>0$ the following holds provided that $\eps>0$ is small enough. For $t\in X_{n,\eps}(C)$, $c\in\Crit$ and $1\le j\le n$,  if $\widetilde{p}_j^{(c)}(t;\eps)>C_0$, then there is a pre-critical $c$-parameter box $B(t_*, r)$ of order $S_j^{(c)}(t;\eps)$ in $\mathcal{M}_{n,\lambda}^{(c)}(\eps)$ such that
$$\dep(t|B(t_*, r))\ge\widetilde{p}_j^{(c)}(t;\eps)-C_0.$$
\end{lem}

\begin{proof}
Fix $C$. By Corollary~\ref{cor:Xnepsbound} and Proposition~\ref{prop:parabox},
provided that $\eps>0$ is small enough, there is a $\lambda$-bounded $c$-parameter box $B(t, r_0)$ of order $S_j:=S_j^{(c)}(t,\eps)$, with $r_0\A(f_t(c), t, S_j)=\theta,$ where $\theta>0$ is a constant independent of $C$.
Let $c'$ be the critical point such that $f_t^{S_j+1}(c)\in\tB(c';\eps)$, $p_j=p_j^{(c)}(t,\eps)$ and $d_j=d_j^{(c)}(t,\eps)$.
Assume that $p_j$ and $d_j$ are large.
Since $|(\xi_{S_j}^{(c)}(t))'|\cdot r_0\asymp |Df_t^{S_j} (f_t(c))| \cdot r_0\asymp e^{p_j}\cdot |f_t^{S_j+1}(c)-c'|$,
there exists a $\lambda$-parameter box $B(t_*, r_*)\subset B(t,r_0)$ of order $S_j$ such that
$r_*\asymp r_0$, $t_*\in\sC_{S_j}^{(c)}$ and $\dep(t|B(t_*,r_*))-p_j$ is bounded away from $-\infty$.
Let $r=r_\lambda(t_*,\eps)$. Clearly, $B(t_*, r)\in \mathcal{M}_{n,\lambda}^{(c)}(\eps)$.
If $r\ge r_*$ then $\dep (t|B(t_*, r))\ge \dep (t|B(t_*, r_*))$ and we are done. So assume
$r<r_*$.

We claim that $\partial(\xi^{(c)}_{S_j}(B(t_*, r)))\cap\partial\tB(c',\eps)\neq\emptyset$. Otherwise we would have $r=\eps$. Since $t\in X_{n,\eps}(C)$, by Lemma~\ref{lem:Dfnq}, we would have that $|Df_{t_1}^{S_j}(f_{t_1}(c))|$ were much bigger than $(D_{c'}(\eps))^{-1}$. It would then follow that $\xi^{(c)}_{S_j}(B(t_*,r))\supsetneq \tB(c',\eps)$, contradicting the definition of $r$.

By the bounded distortion property of $\xi^{(c)}_{S_j}|B(t_*,r)$, it follows that $\xi^{(c)}_{S_j}(B(t_*,r))\supset \tB(c',\eps')$ holds for some $\eps'\asymp \eps.$
Since $|f_t^{S_j+1}(c)-c'|\geq e^{-d_j}|\widetilde{B}(c';\eps)|$,
we conclude that $\dep (t|B(t_*,r))-d_j$ is bounded away from $-\infty$. The lemma is proved.
\end{proof}

\begin{proof}[Proof of the Reduced Main Theorem]
(i)  Let $\lambda>1$ and $C_0$ be as above. We may assume $C>8C_0$.
Consider $t\in X_{n,\eps}(C)\setminus X_{n+1,\eps}(C)$ with $\eps>0$ small.
By Proposition~\ref{prop:totaldepth} (i) (taking $\gamma=1/2$) and Lemma~\ref{lem:shiftparabox}, there exists $c\in\Crit$ such that
\begin{equation}\label{eqn:depMn}
\sum_{B\in\mathcal{M}_{n,\lambda}^{(c)}(\eps)} \dep (t|B)\ge \sum_{k\in \widehat{\mathcal{T}}_\ess^{(c)}(C_0,t;\eps), k\le n} \left(\widetilde{p}^{(c)}_k(t;\eps)-C_0\right)\ge  Cn/4.
\end{equation}
By Lemma~\ref{lem:parapuzzlecombinatorics} and Lemma~\ref{lem:measureest} (taking $\kappa=1/2$), it follows that
$$|X_{n,\eps}(C)\setminus X_{n+1,\eps}(C)|\le \sum_{c\in\Crit}K^n e^{-Cn/8}\left|\supp (\mathcal{M}_{n,\lambda}^{(c)}) \right| \le \left(4\eps\# \Crit\right)\cdot K^n e^{-Cn/8},$$
where $K$ is a constant.

(ii) It follows from Corollary~\ref{cor:Xnepsbound} and Lemma~\ref{prop:parabox}.

(iii) Fix $C>0$, $\tau>\tau_0>1$ and $\kappa>0$. Let $\gamma\in (0,1)$ be a constant such that $\gamma^3 \tau >\tau_0$.
By Proposition~\ref{prop:totaldepth}(ii), for each $t\in Y^m_{\eps}(C,\tau)\setminus Y^{m+1}_{\eps}(C,\tau)$, there exist $c\in\Crit$ and
$n\in\widehat{\mathcal{T}}_\ess^{(c)}(C_0,t;\eps)$ such that $m=S_n^{(c)}(t;\eps)$ and
$$p_n^{(c)}(t;\eps) \ge \gamma d_n^{(c)}(t;\eps)\ge \gamma \tau \log (S_n^{(c)}(t;\eps)+1).$$
We may certainly assume that $Y^m_{\eps}(C,\tau)\setminus Y^{m+1}_{\eps}(C,\tau)\not=\emptyset$. Then $$m\ge S(\eps):=\inf\{S_1^{(c)}(t;\eps): |t|\le \eps, c\in\Crit\}$$
is large provided that $\eps>0$ is small. 
By Lemma~\ref{lem:shiftparabox} there exists $t_*\in \sC_m^{(c)}$ such that
$$\dep (t| B(t_*, r_\lambda(t_*,\eps)))\ge \gamma \tau \log (m+1)-C_0\ge \gamma^2 \tau \log (m+1).$$

Since these parameter boxes $B(t_*, r_\lambda(t_*,\eps))$, $t_*\in \sC_m^{(c)}$, are pairwise disjoint, it follows that
$$|Y^m_{\eps}(C,\tau)\setminus Y^{m+1}_{\eps}(C,\tau)|\le C_1 \eps \# \Crit  \left(m+1\right)^{-\gamma^3\tau}\eps\le \sigma \eps (m+1)^{-\tau_0},$$
since $m$ is large.
\end{proof}

\bibliographystyle{alpha}

\begin{thebibliography}{BRLSvS08}
\bibitem[A]{A}
A. Avila.
\newblock Infinitesimal perturbations of rational maps.
\newblock {\em Nonlinearity} \textbf{15} 695每704 (2002).

\bibitem[AM03] {AM03}
A. Avila and C. G. Moreira.
\newblock Statistical poperties of unimodal maps: smooth families with negative Schwarzian devirative.
\newblock {\em Ast\'erisque} \textbf{268} 81-118 (2003)

\bibitem[AM05]{AM05}
A. Avila and C. G. Moreira.
\newblock Statistical properties of unimodal maps: the quadratic family.
\newblock {\em Ann. of Math. (2)} \textbf{161} 831每881 (2005).

\bibitem[BC1] {BC1}
M. Benedicks and L. Carleson.
\newblock On iterations of $1-ax^2$ on $(-1,1)$.
\newblock {\em Ann. Math. (2)} \textbf{122}, 1-24 (1985)

\bibitem[BC2] {BC2}
M. Benedicks and L. Carleson.
\newblock The dynamics of the H\'enon maps.
\newblock {\em Ann. Math. (2)} \textbf{133}, 73-169 (1991)


\bibitem[CE]{CE}
P. Collet and J. P. Eckmann.
\newblock Positive Lyapunov exponents and absolute continuity for maps of the interval.
\newblock {\em Ergod. Th. Dynam. Sys.} \textbf{3}, 13-46 (1983)

\bibitem[J]{J}
M. Jakobson.
\newblock Absolutely continuous invariant measures for one-parameter families of one-dimensional maps.
\newblock {\em Comm. Math. Phys.} \textbf{81}, 39-88 (1981)

\bibitem[Le1]{Le1}
G. Levin.
\newblock On an analytic approach to the Fatou conjecture.
\newblock {\em Fund. Math.} \textbf{171}, 177-196 (2002).

\bibitem[Le2]{Le2}
G. Levin.
\newblock Perturbations of weakly expanding posrcritical orbits.
\newblock {\em Preprint} 2011.

\bibitem[Le3] {Le3}
G. levin.
\newblock Multipliers of periodic orbits in spaces of rational maps.
\newblock {\em Ergod. Th. Dynam. Sys. } \textbf{31}, 197-243 (2011)

\bibitem[Lu]{Lu}
S. Luzzatto.
\newblock {\em Bounded recurrence of critical points and Jakobson's theorem.}
\newblock The Mandelbrot set, theme and variations, 173每210,
London Math. Soc. Lecture Note Ser., \textbf{274}, Cambridge Univ. Press, Cambridge, (2000).

\bibitem[Ly1]{Ly1}
M. Lyubich.
\newblock Dynamics of quadratic polynomials. III. Parapuzzle and SBR measures.
\newblock {\em Ast\'erisque} \textbf{261} 173每200 (2000).

\bibitem[Ly2]{Ly2}
M. Lyubich.
\newblock  Almost every real quadratic map is either regular or stochastic.
\newblock {\em Ann.  Math.} \textbf{156} 1 - 78(2002).


\bibitem[NV] {NV}
T. Nowicki and S. van Strien.
\newblock Invariant measures exist under a summability conditon for unimodal maps
\newblock {\em Invent. Math. } \textbf{105}, 123-136 (1991)

\bibitem[R]{R}
M. Rees.
\newblock Positive measure sets of ergodic rational maps.
\newblock{\em Ann. Sci. \'ecole Norm. Sup. (4)} \textbf{19} 383每407 (1986).

\bibitem[S] {S}
W. Shen.
\newblock On stochastic stability of non-uniformly expanding interval maps.
\newblock {\em ArXiv:1107.2537.}


\bibitem[TTY]{TTY}
P. Thieullen,  C. Tresser and L.-S. Young.
\newblock Positive Lyapunov exponent for generic one-parameter families of unimodal maps.
\newblock {\em J. Anal. Math.} \textbf{64} 121每172 (1994).

\bibitem[T] {T}
M. Tsujii.
\newblock Positive Lyapunov exponent in families of one demensional dynamical systems.
\newblock {\em Invent. Math. } \textbf{111}, 113-137 (1993)

\bibitem[V]{V}
M. Viana.
\newblock Multidimensional nonhyperbolic attractors
\newblock {\em Publ. Math. IHES.} \textbf{85} 63-96 (1997).


\bibitem[WT]{WT}
Q. Wang and H. Takahasi.
\newblock Nonuniformly expanding 1d maps with logarithmic singularities.
\newblock {\em ArXiv:1106.1707.}


\bibitem[Yo]{Yo}
J.C. Yoccoz.
\newblock Jakobson's Theorem.
Manuscript 1997.
\end{thebibliography}

\end{document}